\RequirePackage{fix-cm}
\documentclass[smallextended]{svjour3}       
\smartqed  
\usepackage{graphicx}
\usepackage{url}
\usepackage{amssymb}
\usepackage{amsmath}
\usepackage{color}
\newcommand{\rmd}{\mathrm{d}}

\newtheorem{fact}{Fact}

\begin{document}

\title{Characterization of manifolds of constant curvature by ruled surfaces\footnote{This is a pre-print of an article published in São Paulo Journal of Mathematical Sciences (2022). The final authenticated version is available online at \url{https://doi.org/10.1007/s40863-022-00319-7}.}}

\author{Luiz C. B. da Silva         \and
        Jos\'e D. da Silva 
}


\institute{Da Silva, L. C. B. \at Department of Physics of Complex Systems,\\ 
Weizmann Institute of Science,              Rehovot 7610001, Israel\\
         \email{luiz.da-silva@weizmann.ac.il}
           \and 
           Da Silva, J. D. \at
             Departamento de Matem\'atica,\\
             Universidade Federal Rural de Pernambuco, 
         52171-900, Recife, Pernambuco, Brazil\\
         \email{jose.dsilva@ufrpe.br}
}
\date{\today}

\maketitle

\begin{abstract}
We investigate ruled surfaces in 3d Riemannian manifolds, i.e., surfaces foliated by
geodesics. In 3d space forms, we find the striction curve, distribution parameter, and
the first and second fundamental forms, from which we obtain the Gaussian and
mean curvatures. We also provide model-independent proof for the known fact that
extrinsically flat surfaces in space forms are ruled. This proof allows us to identify
the necessary and sufficient condition the curvature tensor must satisfy for an extrinsically
flat surface in a generic 3d manifold to be ruled. Further, we show that if a 3d
manifold has an extrinsically flat surface tangent to any 2d plane and if they are all
ruled surfaces, then the manifold is a space form. As an application, we prove that
there must exist extrinsically flat surfaces in the Riemannian product of the hyperbolic
plane, or sphere, with the reals, and that do not make a constant angle with the
real direction.    
\end{abstract}

\section{Introduction}

Special classes of surfaces play a fundamental role in differential geometry, they help visualize key concepts and often allows us to test conjectures. Of particular importance are \emph{ruled surfaces}, i.e., surfaces obtained from a 1-parameter family of geodesics. A remarkable result in the  three-dimensional (3d) Euclidean space $\mathbb{E}^3$ asserts that every flat surface is ruled. Indeed, flat surfaces are generically described as a union of open pieces of tangent surfaces, cones, and cylinders \cite{DoCarmo2016DG,Kreyszig1991}. In general, the Gaussian curvature {of a ruled surface is} always non-positive and depends explicitly on the so-called distribution parameter \cite{DoCarmo2016DG}, which relates the tangents of the generating curve and the vector field giving the direction of the family of geodesics to the derivative of this vector field in the direction of the generating curve. By introducing the concept of striction curves, singular points on Euclidean ruled surfaces are allowed and, when they exist, they are located along the striction curve \cite{DoCarmo2016DG}.  Despite that ruled surfaces can be defined on any manifold, to the best of our knowledge, no systematic treatment of the fundamentals of these surfaces have been done, even in space forms. Indeed, except for Euclidean space, ruled surfaces are generally used only as a tool to obtain special examples of surfaces, such as minimal surfaces in space forms \cite{doCarmoTAMS1983,LawsonAM1970}, understand certain families of curves in the 3-sphere \cite{HuangJGP2019}, constant angle surfaces in the product manifolds $\mathbb{S}^2\times\mathbb{R}$ and $\mathbb{H}^2\times\mathbb{R}$ \cite{D+07,DM09}, the characterization of complete flat surfaces in $\mathbb{H}^2\times\mathbb{R}$ \cite{BarbosaJG2020} and space forms \cite{HondaTMJ2012,PortnoyPJM1975}, and surfaces of finite type in the 3-sphere \cite{KimMathMDPI2020}, just to name a few.  To fill this gap, here we systematically study ruled surface in space forms. More precisely, we shall compute their first and second fundamental forms, Gaussian and mean curvatures, and we also introduce the concept of distribution parameter and striction curves.

After establishing the fundamentals of ruled surfaces in space forms, we provide a model-independent proof of the fact that extrinsically flat surfaces in space forms must be ruled (Theorem \ref{thr::ExtFlatImpliesRuledInSpcForm}). To the best of our knowledge, this is the first proof of this type. More than just an alternative proof, this  allows us to understand why the theorem is true in constant curvature, which later leads us to the main result of this work ({Theorem \ref{Theor::CharSpcFraRuledSurf}}):
\newline
\newline
\textbf{Main Theorem:} \textit{Let $M^3$ be {a connected} 3d manifold with the property that there exists an extrinsically flat surface tangent to any 2d plane and that every extrinsically flat surface in $M^3$ is ruled. Then, $M^3$ must have constant sectional curvature.}
\newline

As an application, we use our Main Theorem to prove that there must exist extrinsically flat surfaces in the product manifolds $\mathbb{H}^2\times\mathbb{R}$ and $\mathbb{S}^2\times\mathbb{R}$ that do not make a constant angle with the $\mathbb{R}$-direction (Theorem \ref{Thr::ExistsExtFlatSurfNonCteAngleInH2RandS2R}). 

The remaining of this work is divided as follows. After preliminaries concerning Riemannian space forms in Sect. 2, we establish in Sect. 3 the fundamentals of ruled surfaces. We compute  the first and second fundamental forms in Subsect. 3.1., characterize helicoids as the only minimal ruled surfaces in Subsect. 3.2., compute the striction curve in Subsect. 3.3., which together with the distribution parameter introduced in Subsect. 3.4. allow us to locate the singular points of ruled surfaces. With the distribution parameter, we also show that extrinsically flat ruled surfaces are generically described as a union of open pieces of tangent surfaces, cones, and cylinders. It is worth mentioning that the computations in Sect. 3 are all model-dependent. {The final answers are model-independent, though.} In Sect. 4, we adopt a model-independent approach and prove that the implication ``extrinsically flat $\Rightarrow$ ruled" is essentially a characteristic property of constant curvature. We also apply these results to extrinsically flat surfaces in $\mathbb{H}^2\times\mathbb{R}$ and $\mathbb{S}^2\times\mathbb{R}$.  Finally, we present our concluding remarks in the last section.

\section{Preliminaries}

Let  $\mathbb{E}^{4}$ be the $4$-dimensional Euclidean space, i.e., $\mathbb{R}^{4}$ equipped with the standard Euclidean metric $\langle X,Y\rangle_e=\sum_{i=1}^{4}x_iy_i$, and  $\mathbb{E}_1^{4}$ be the 4d Lorentzian space, i.e., $\mathbb{R}^{4}$ equipped with the pseudo-metric  $\langle X,Y\rangle_1=-x_1y_1+\sum_{i=2}^{4}x_iy_i$. In this work, we will be primarily interested in the 3d sphere $\mathbb{S}^{3}(r)$ and in the 3d hyperbolic space $\mathbb{H}^{3}(r)$ modeled as submanifolds of $\mathbb{E}^{4}$ and $\mathbb{E}_1^{4}$, respectively. More precisely, consider
$$
\mathbb{S}^{3}(r) = \{q\in\mathbb{R}^{4}:\langle q,q\rangle_e=r^2\}
\mbox{ and }
\mathbb{H}^{3}(r) = \{q\in\mathbb{R}^{4}:\langle q,q\rangle_1=-r^2,q_1>0\}
$$
equipped with the metric induced from the corresponding ambient space. We shall denote the induced metric by $\langle\cdot,\cdot\rangle$. (The context will make clear whether the induced inner product $\langle\cdot,\cdot\rangle$ comes from $\langle\cdot,\cdot\rangle_e$ or $\langle\cdot,\cdot\rangle_1$.) 


Denoting by $\nabla$ the Levi--Civita connection on $\mathbb{S}^{3}(r)$ (or $\mathbb{H}^{3}(r)$) and by $\nabla^0$ the connection on $\mathbb{E}^{4}$ (or $\mathbb{E}^{4}_1$, respectively), they are related by the Gauss formula as
\begin{equation}
(\nabla_{X}\,Y)(q)=(\nabla^0_{X}\,Y)(q)\pm\frac{1}{r^2}\langle X,Y\rangle_q\,q,\label{eq::CovDerAndUsualDer}
\end{equation}
where $q$ is the position vector, i.e., the canonical immersion $q:\mathbb{S}^{3}(r)\to\mathbb{E}^{4}$ for the plus sign and $q:\mathbb{H}^{3}(r)\to\mathbb{E}^{4}_1$ for the minus sign. In these models, the exponential map at $p\in\mathbb{S}^{3}(r)$ and  $p\in\mathbb{H}^{3}(r)$ are given by
\[
\mbox{exp}_p(vZ)=\cos\left(\frac{v}{r}\right)p+r\sin\left(\frac{v}{r}\right)Z
\mbox{ and }
\mbox{exp}_p(vZ)=\cosh\left(\frac{v}{r}\right)p+r\sinh\left(\frac{v}{r}\right)Z,
\]
where $Z\in \mathbb{S}^2(1)\subset T_p\mathbb{S}^{3}(r)$ and  $Z\in\mathbb{S}^2(1)\subset T_p\mathbb{H}^{3}(r)$, respectively. 

In $\mathbb{E}^4$ and $\mathbb{E}_1^4$ it is possible to introduce a ternary product between vectors $u,v,w$ according to the formal determinants
\begin{equation}
    u\times v\times w = \left\vert
    \begin{array}{cccc}
        u_1 & u_2 & u_3 & u_4\\
        v_1 & v_2 & v_3 & v_4\\
        w_1 & w_2 & w_3 & w_4\\
        \mathbf{e}_1 & \mathbf{e}_2 & \mathbf{e}_3 & \mathbf{e}_4 \\
    \end{array}
    \right\vert
\mbox{ and }u\times_1 v\times_1 w = \left\vert
    \begin{array}{cccc}
        u_1 & u_2 & u_3 & u_4\\
        v_1 & v_2 & v_3 & v_4\\
        w_1 & w_2 & w_3 & w_4\\
        -\mathbf{e}_1 & \mathbf{e}_2 & \mathbf{e}_3 & \mathbf{e}_4 \\
    \end{array}
    \right\vert,
\end{equation}
where $\{\mathbf{e}_1,\mathbf{e}_2,\mathbf{e}_3,\mathbf{e}_4\}$ is the canonical basis of $\mathbb{R}^4$. These ternary products allow us to introduce vector products in $\mathbb{S}^3(r)$ and $\mathbb{H}^3(r)$ by taking the third vector in the ternary product  as the position vector. In other words, given  $X,Y\in T_q\mathbb{S}^3(r)$ [or $X,Y\in T_q\mathbb{H}^3(r)$], the vector product between them is the unique tangent vector $X\times Y$ [or $X\times_1Y$] satisfying for all $W\in T_q\mathbb{S}^3(r)$ [or $W\in T_q\mathbb{S}^3(r)$] the relation
\begin{equation}\label{def::VectorProduct}
    \langle X\times Y,W\rangle = \det(q,X,Y,W) \,\,[\mbox{or }\langle X\times_1 Y,W\rangle_1 = \det(q,X,Y,W)].
\end{equation}
When the point $q$ is clear from the context, we shall write $X\times Y=X\times Y\times q$ and $X\times_1 Y=X\times_1 Y\times_1 q$.

Let $M^{3}$ be a Riemannian manifold of class at least $C^3$ with metric $\langle\cdot,\cdot\rangle$, Levi--Civita connection $\nabla$, and Riemann curvature tensor
$$R(X,Y)W=\nabla_Y\nabla_XW-\nabla_X\nabla_YW+\nabla_{[X,Y]}W,$$ where $X,Y$, and $W$ are tangent vector fields in $M^3$. If $X,Y\in T_pM^3$, the \textit{sectional curvature} of the plane $\mbox{span}\{X,Y\}\vert_p\subset T_pM^3$ at $p$ is 
\begin{equation}
K_p(X,Y)=\frac{\langle R(X,Y)X,Y\rangle_p}{\langle X,X\rangle_p\langle Y,Y\rangle_p-\langle X,Y\rangle_p^2}.
\end{equation}
We say that $M^3$ is a \emph{Riemannian space form} if it has constant sectional curvature $K_0$. In addition, $M^3$ is locally isometric to a sphere $\mathbb{S}^{3}(r)$ if $K_0=\frac{1}{r^{2}}$, to an Euclidean space $\mathbb{E}^{3}$ if $K_0=0$, or to a hyperbolic space $\mathbb{H}^{3}(r)$ if $K_0=-\frac{1}{r^{2}}$  \cite{doCarmo1992,Spivak1979v4}. In terms of the curvature tensor $R_{ijk\ell}=\langle R(e_i,e_j)e_k,e_{\ell}\rangle$, where  $\{e_i\}$ is an orthonormal basis of $T_pM^3$, and $\delta_{ij}$ is the usual Kronecker delta tensor, a manifold is a  space form if, and only if, at any $p\in M^3$ it is valid $R_{ijk\ell}=K_0(\delta_{ik}\delta_{j\ell}-\delta_{i\ell}\delta_{jk})$.   

As it is well known, in Euclidean space the intrinsic $K_{int}$ and extrinsic $K_{ext}$ Gaussian curvatures of a surface coincide, i.e., the Gaussian curvature computed from the induced metric is equal to the product of the principal curvatures. In a Riemannian manifold this is no longer true, instead the intrinsic and extrinsic curvatures differ by the value of the sectional curvature of the tangent plane of the surface. More precisely, if $\Sigma^2\subset M^{3}$ is an oriented surface with unit normal $\xi$, the \textit{shape operator}  of $\Sigma^2$ is defined by
$
S_p(X)=-(\nabla_X\xi)(p).
$ Since $S_p$ is symmetric with respect to $\langle\cdot,\cdot\rangle$, $S_p$ has $2$ eigenvalues $\kappa_1,\kappa_2$, known as the \textit{principal curvatures}.  The \emph{extrinsic Gaussian curvature} $K_{ext}$ of $\Sigma^2$ is defined as $K_{ext}=\kappa_1\kappa_2$, while the \emph{intrinsic Gaussian curvature} $K_{int}$ of $\Sigma^2$ is defined as the sectional curvature computed from the induced metric of $\Sigma^2$. Now, if $K_{sec}$ is the sectional curvature of $T_p\Sigma^2\subset T_pM^3$, then the Gauss Theorema Egregium for $\Sigma^2$ in $M^3$ reads \cite{doCarmo1992,Spivak1979v4}
\begin{equation}
    K_{int} - K_{ext} = K_{sec}.
\end{equation}
Thus, in a space form the two Gaussian curvatures differ by a constant.

\section{Ruled surfaces in the sphere and hyperbolic space}

In this section, we compute the coefficients of the first and second fundamental forms of ruled surfaces from which we compute the Gaussian and mean curvatures. The expression for the extrinsic Gaussian curvature {assumes} its simplest form after introducing the concepts of striction curve and distribution parameter. The striction curves allow us to understand when, and where, ruled surfaces have singularities, i.e., points where the metric degenerates, while the distribution parameter helps in finding a (local) classification of extrinsically flat ruled surfaces.

{Without loss of generality, we may do the explicit calculations for $\mathbb{S}^3(r)$ only. For surfaces in $\mathbb{H}^3(r)$, we can proceed in an analogous way by substituting the trigonometric functions by their hyperbolic relatives. 
}

\subsection{First and second fundamental forms}

Let $\alpha :I\rightarrow\mathbb{S}^3(r)$ be a regular curve and  $Z(u)$ be a unit vector field along $\alpha(u)$ such that $\forall u\in I\subset\mathbb{R}$, $\nabla_{\alpha'}Z\neq0$, i.e., the ruled surface is not a cylinder. A ruled surface $\Sigma^2$ with directrix (or generating curve)  $\alpha$ and rulings tangent to $Z$ is parametrized as
\begin{equation}
    \psi(u,v)=\text{exp}_{\alpha(u)}(vZ(u))=\cos{\left(\dfrac{v}{r}\right)}\alpha(u)+r\sin{\left(\dfrac{v}{r}\right)}Z(u).
\end{equation}
Since $\gamma_u(v)=\exp_{\alpha(u)}(vZ(u))$ is a geodesic with initial velocity $Z(u)$ at $\alpha(u)$, then the parallel transport of $Z$ along the rulings, denoted by $PZ$, is simply given by
\begin{equation}\label{eq::PZ}
    (PZ)(u,v)=\frac{\partial\psi}{\partial v}(u,v)=-\frac{1}{r}\sin\left(\frac{v}{r}\right)\alpha(u)+\cos\left(\frac{v}{r}\right)Z(u).
\end{equation}

The coefficients of the \emph{first fundamental form} will be denoted by $g_{11}=\langle\psi_u,\psi_u\rangle$, $g_{12}=\langle\psi_u,\psi_v\rangle$, and $g_{22}=\langle\psi_v,\psi_v\rangle$. On the other hand, the coefficients of the \emph{second fundamental form} are $h_{11}=\langle\nabla_{\psi_u}\psi_{u},\xi\rangle$, $h_{12}=\langle\nabla_{\psi_u}\psi_{v},\xi\rangle$, and $h_{22}=\langle\nabla_{\psi_v}\psi_{v},\xi\rangle$, where $\xi=\frac{1}{\sqrt{g}}\psi_u\times\psi_v$  denotes the unit normal, $g=\det g_{ij}=\Vert\psi_u\times\psi_v\Vert^2$, and $\times$ is the vector product between tangent vectors on the sphere defined previously.

The tangent planes are spanned by the tangent vectors
\begin{equation}\label{eq::TangentVecPsi}
    \psi_u = c_v\alpha'+rs_vZ'=c_v\alpha'+rs_vDZ-\frac{s_v}{r}\langle\alpha',Z\rangle\alpha\mbox{ and }\psi_v = PZ,
\end{equation}
where we adopted the shorthand notation $c_v=\cos(\frac{v}{r})$, $s_v=\sin(\frac{v}{r})$, and $D=\nabla_{\alpha'}$. Then, the coefficients of the first fundamental form are
\begin{eqnarray}
    g_{11} & = &\langle c_v\alpha'+rs_vDZ,c_v\alpha'+rs_vDZ\rangle+\frac{s_v^2}{r^2}\langle\alpha',Z\rangle^2\langle\alpha,\alpha\rangle \nonumber\\
    & = &c_v^2\langle\alpha',\alpha'\rangle+2rs_vc_v\langle\alpha',DZ\rangle+r^2s_v^2\langle DZ,DZ\rangle+s_v^2\langle\alpha',Z\rangle^2,
\end{eqnarray}
\begin{equation}
    g_{22} = 1,\mbox{ and } g_{12}=\frac{s_v^2}{r^2}\langle\alpha',Z\rangle\langle\alpha,\alpha\rangle+c_v^2\langle\alpha',Z\rangle+rs_vc_v\langle Z,DZ\rangle = \langle\alpha',Z\rangle.
\end{equation}

To compute the second fundamental form, we need the following lemma. It will prove useful in finding the Gaussian and mean curvatures, since one is able to perform the non-trivial computations  along the generating curve $\alpha$ only. (Note that we should see Eq. \eqref{Eq::fromAWDwToAwDw} below as an auxiliary identity in $\mathbb{R}^4$ since $Z$ and $PZ$ are tangent vectors on distinct points of $\mathbb{S}^3(r)$, {or $\mathbb{H}^3(r)$}.) 

\begin{lemma}\label{lemma::fromAWDwToAwDw}
Let $PZ$ denote the parallel transport of $Z$ along the rulings of the ruled surface $\Sigma^2\subset\mathbb{S}^3(r)$ {or $\Sigma^2\subset\mathbb{H}^3(r)$}. If $X$ and $Y$ are two vector fields in $\mathbb{S}^3(r)$ {or $\mathbb{H}^3(r)$}, then in {$\mathbb{R}^4$} it is valid that
\begin{equation}\label{Eq::fromAWDwToAwDw}
{    \det(X,Y,PZ,\psi) =\frac{1}{c_v}\det(X,Y,Z,\psi),   }
\end{equation}
{where we have adopted the shorthand notation $c_v=\cos(\frac{v}{r})$ in $\mathbb{S}^3(r)$ and $c_v=\cosh(\frac{v}{r})$ in $\mathbb{H}^3(r)$.}
\end{lemma}
\begin{proof}
Let $X$ and $Y$ be two vector fields in $\Sigma^2\subset\mathbb{S}^3(r)$. If we write $PZ=A_1\alpha+B_1Z$ and $\psi=A_2\alpha+B_2Z$, where $A_1=-\frac{1}{r}s_v,B_1=A_2=c_v$, and $B_2=rs_v$, then
\begin{eqnarray}
    \det(X,PZ,Y,\psi) 
    &=& A_1B_2\det(X,\alpha,Y,Z)+B_1A_2\det(X,Z,Y,\alpha)\nonumber\\
    &=& -\frac{A_1B_2}{A_2}\det(X,Z,Y,A_2\alpha)+B_1\det(X,Z,Y,A_2\alpha)\nonumber\\
    &=& \frac{B_1A_2-A_1B_2}{A_2}\det(X,Z,Y,A_2\alpha+B_2Z).\nonumber
\end{eqnarray}
Finally, substituting for $A_i$ and $B_i$ gives the desired result.
\qed \end{proof}

\begin{proposition}\label{Prop::2ndffRulSurf}
The coefficients $h_{ij}$ of the second fundamental form of a ruled surface $\Sigma^2\subset\mathbb{S}^3(r)$, {or $\Sigma^2\subset\mathbb{H}^3(r)$,} are
\[
    h_{11} = \frac{c_v(D\alpha',\alpha',Z)+rs_v[(D\alpha',DZ,Z)
     +(D^2Z,\alpha',Z)]}{\sqrt{g}}+\frac{r^2s^2_v(D^2Z,DZ,Z)}{c_v\sqrt{g}},
\]
\[
h_{12} = \frac{1}{c_v\sqrt{g}}(\alpha',Z,DZ),\mbox{ and } h_{22}=0,    
\]
where $g=\det g_{ij}$ is the determinant of the metric, $(X,Y,W)=\langle X, Y\times W\rangle$ is the mixed product, {$s_v=\sin(\frac{v}{r})$ and $c_v=\cos(\frac{v}{r})$ in $\mathbb{S}^3(r)$, $s_v=\sinh(\frac{v}{r})$ and $c_v=\cosh(\frac{v}{r})$ in $\mathbb{H}^3(r)$,} and $D=\nabla_{\alpha'}$.
\end{proposition}
\begin{corollary}\label{Cor::KextRuledSurf}
Let $\Sigma^2\subset\mathbb{S}^3(r)$, {or $\Sigma^2\subset\mathbb{H}^3(r)$,} be a ruled surface with directrix $\alpha$ and rulings tangent to $Z$. Then, the extrinsic curvature of $\Sigma^2$ is given by
\begin{equation}\label{Kextsuperfdesenv}
    K_{ext}(u,v)=-\left[\frac{1}{c_v}\dfrac{(\alpha'(u),Z(u),\nabla_{\alpha'(u)}Z(u))}{g_{11}g_{22}-(g_{12})^2}\right]^2\leq0,
\end{equation}
{where $c_v=\cos(\frac{v}{r})$ in $\mathbb{S}^3(r)$ and $c_v=\cosh(\frac{v}{r})$ in $\mathbb{H}^3(r)$.}
\end{corollary}

\begin{proof}[of Prop. \ref{Prop::2ndffRulSurf}]
Let us first compute the unit normal $\xi$ of $\Sigma^2$. We have
\begin{eqnarray}
\psi_u\times\psi_v
&=&c_v\alpha'\times PZ\times\psi+rs_vDZ\times PZ\times\psi-\dfrac{\langle PZ,\alpha'\rangle}{r}s_v\psi\times PZ\times\psi\nonumber\\
&=&c_v\alpha'\times PZ\times\psi+rs_vDZ\times PZ\times\psi.
\end{eqnarray}
Note that the last expression is not necessarily seen as a linear combination of vectors in $\mathbb{S}^3(r)$, but each term in the sum makes sense in the ambient $\mathbb{E}^4$.

\textsc{Coefficient} $h_{22}$: Since $\psi_v$ is the tangent vector field along the geodesic $\gamma(v)=\psi(\cdot ,v)$, it follows that $\nabla_{\psi_v}{\psi_v}=0$ and, therefore, $h_{22}=0$.

\textsc{Coefficient} $h_{12}$: Using the Gauss formula, Eq. (\ref{eq::CovDerAndUsualDer}), we compute $\nabla_{\psi_v}{\psi_u}$ as
\begin{equation}
\nabla_{\psi_v}{\psi_u}
=
\nabla^0_{\psi_v}{\psi_u} +\dfrac{1}{r^2}\langle \psi_u ,\psi_v\rangle\psi(u,v)=  \dfrac{\partial\psi_u}{\partial v}+\dfrac{g_{12}}{r^2}\psi(u,v).
\end{equation}
Now, since $\psi_u=c_v\alpha'+rs_vZ'$, it follows
$$\dfrac{\partial\psi_u}{\partial v}=-\dfrac{1}{r}\sin{\left(\dfrac{v}{r}\right)}\alpha'+\cos{\left(\dfrac{v}{r}\right)}Z'$$
and, therefore, 
\begin{equation}\label{curvaturaextrin}
    \nabla_{\psi_v}{\psi_u}=-\dfrac{1}{r}\sin{\left(\dfrac{v}{r}\right)}\alpha'+\cos{\left(\dfrac{v}{r}\right)}Z'+\dfrac{g_{12}}{r^2}\psi(u,v).
    \end{equation}
    
Using the Gauss formula again, the derivative of $Z$ in $\mathbb{E}^4$ is written as $$Z'=DZ-\dfrac{1}{r^2}\langle\alpha',Z\rangle\psi(u,v),$$ and, therefore, we may rewrite $\nabla_{\psi_v}{\psi_u}$ in Eq. (\ref{curvaturaextrin}) as
\begin{equation}
    \nabla_{\psi_v}\psi_u=-\dfrac{1}{r}s_v\alpha'+c_vDZ+\dfrac{g_{12}-\langle\alpha',Z\rangle c_v}{r^2}\,\psi(u,v).
\end{equation}
Finally, the coefficient $h_{12}=\langle\nabla_{\psi_v}\psi_u,\frac{1}{\sqrt{g}}\psi_u\times\psi_v\rangle$ is
\begin{eqnarray*}
         h_{12}&=& \left\langle-\dfrac{1}{r}s_v\alpha'+c_vDZ, \dfrac{\psi_u\times\psi_v}{\sqrt{g}} \right\rangle \nonumber\\
         & = & \dfrac{1}{\sqrt{g}}\Big\{-\dfrac{1}{r}s_v\left[c_v(\alpha',\alpha',PZ)+rs_v(\alpha',DZ,PZ)\right]+\nonumber\\
         & + &  c_v\left[c_v(DZ,\alpha',PZ)+rs_v(DZ,DZ,PZ)\right]\Big\}\nonumber\\
         & = & \dfrac{1}{\sqrt{g}}\Big[c_v^2(DZ,\alpha',PZ)-s_v^2(\alpha',DZ,PZ)\Big]= \sec\left(\frac{v}{r}\right)\dfrac{(\alpha',Z,DZ)(u)}{\sqrt{g}},
\end{eqnarray*}
where in the last equality we used lemma \ref{lemma::fromAWDwToAwDw} with $X=\alpha'$ and $Y=DZ$.

\textsc{Coefficient} $h_{11}$: First, we have $$\nabla_{\psi_u}\psi_u=\nabla_{\psi_u}^0\psi_u+\dfrac{1}{r^2}\langle\psi_u,\psi_u\rangle\psi=c_v\alpha''+rs_vZ''+\dfrac{g_{11}}{r^2}\psi.$$
Using that $Z'=\nabla_{\alpha'}^0Z=DZ-\dfrac{1}{r^2}\langle\alpha',Z\rangle\psi$ gives
$$
\begin{array}{ccl}
     Z''&=&\nabla_{\alpha'}^0(DZ-\dfrac{1}{r^2}\langle\alpha',Z\rangle\psi)\\
     &=&\left(D^2Z-\dfrac{1}{r^2}\langle\alpha',DZ\rangle\psi\right)-\dfrac{1}{r^2}\left(\langle\alpha',Z\rangle'\psi+\langle\alpha',Z\rangle\nabla_{\alpha'}^0\psi\right)
\end{array}
$$
and, consequently,
\begin{eqnarray}
    \nabla_{\psi_u}\psi_u &=& c_v\alpha''+rs_vD^2Z-rs_v\left[\dfrac{\langle\alpha',DZ\rangle+\langle\alpha',Z\rangle'}{r^2}\psi+\dfrac{\langle\alpha',Z\rangle\nabla_{\alpha'}^0\psi}{r^2}\right]+\dfrac{g_{11}}{r^2}\psi\nonumber\\
    & = & c_vD\alpha'+rs_vD^2Z+A\nabla_{\alpha'}^0\psi+B\psi,\nonumber
\end{eqnarray}
where the coefficients $A$ and $B$ are not important to compute $h_{11}$, since the inner {products} of $\nabla_{\alpha'}^0\psi=\psi_u$ and $\psi$ with $\psi_u\times\psi_v$ vanish. Then, $h_{11}$ is 
\begin{eqnarray}
    h_{11} & = & \left\langle\nabla_{\psi_u}\psi_u,\dfrac{\psi_u\times\psi_v}{\sqrt{g}}\right\rangle\nonumber\\
    & = & \frac{1}{\sqrt{g}}\langle c_vD\alpha'+rs_vD^2Z,c_v\alpha'\times PZ\times\psi+rs_vDZ\times PZ\times\psi\rangle\nonumber\\
    &=&\frac{1}{\sqrt{g}}\Big[c^2_v\det(D\alpha',\alpha',PZ,\psi)+rs_vc_v\det(D\alpha',DZ,PZ,\psi)+\nonumber\\
     &+&rs_vc_v\det(D^2Z,\alpha',PZ,\psi)
     +r^2s^2_v\det(D^2Z,DZ,PZ,\psi)\Big].
\end{eqnarray}
Finally, applying lemma \ref{lemma::fromAWDwToAwDw} in the last equality above, we  get
\[
    h_{11} = \frac{c_v(D\alpha',\alpha',Z)+rs_v[(D\alpha',DZ,Z)
     +(D^2Z,\alpha',Z)]}{\sqrt{g}}+\frac{r^2s^2_v(D^2Z,DZ,Z)}{c_v\sqrt{g}}.
\]
\qed \end{proof}
\begin{remark}
If $r\gg1$, then {$\cos(\frac{v}{r}),\cosh(\frac{v}{r})\approx1$} and {$\sin(\frac{v}{r}),\sinh(\frac{v}{r})\approx\frac{v}{r}$}. Thus, the coefficients of the second fundamental form behave as $h_{22}=0$, $h_{12}\approx \frac{1}{\sqrt{g}}(\alpha',Z,DZ)$, and 
$$h_{11}\approx\frac{(D\alpha',\alpha',Z)+v[(D\alpha',DZ,Z)+(D^2Z,\alpha',Z)]+v^2(D^2Z,DZ,Z)}{\sqrt{g}}.$$
Therefore, we recover the known results from Euclidean space \cite{Barbosa1986,DoCarmo2016DG}.
\end{remark}

\subsection{Ruled minimal surfaces}

\begin{figure}[t]
    \centering
    \includegraphics[width=0.75\linewidth]{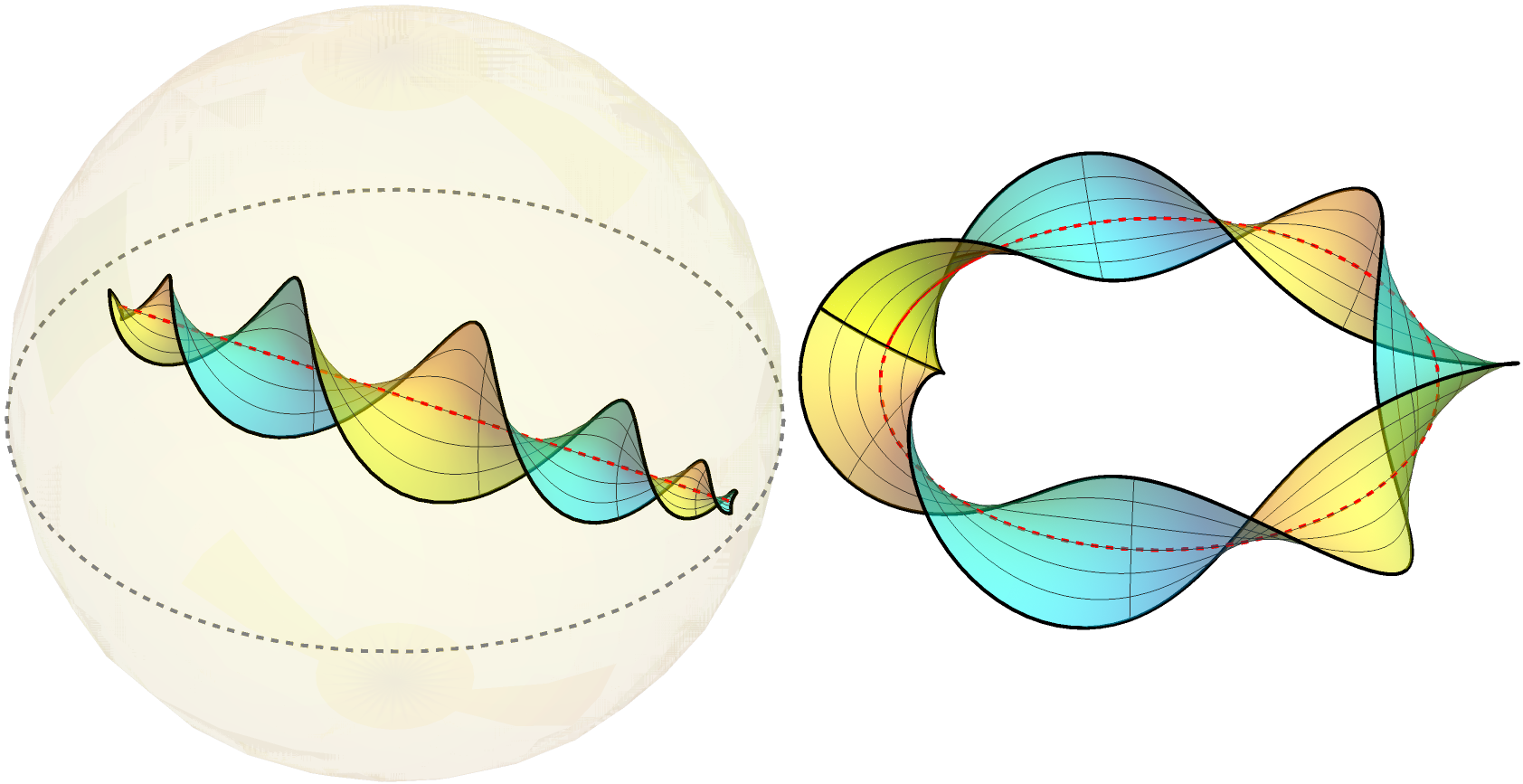}
    \caption{Helicoids in space forms: (Left) Helicoid in the Poincar\'e ball model of $\mathbb{H}^3(1)$ obtained from Eq. \eqref{eq::H3helicoids} through stereographic projection $(x_1,x_2,x_3,x_4)\mapsto (\frac{x_2}{1+x_1},\frac{x_3}{1+x_1},\frac{x_4}{1+x_1})$. The dashed gray line indicates the points $(x,y,z)$ in the ideal boundary with $z=0$; (Right) Helicoid in $\mathbb{S}^3(1)$ obtained from Eq. \eqref{eq::S3helicoids} through stereographic projection $(x_1,x_2,x_3,x_4)\mapsto (\frac{x_1}{1-x_4},\frac{x_2}{1-x_4},\frac{x_3}{1-x_4})$. (In the figures, the dashed red lines indicate the generating curves, we set $\omega=3.0$, and both surfaces have constant width, i.e., $\vert v\vert\leq\mbox{const.}$. Note that in $\mathbb{S}^3$ helices are periodic if and only if $\omega\in\mathbb{Q}$ \cite{ChakrabartiInvolve2019}.) {Figures generated with Mathematica.}}
    \label{fig:Helicoids}
\end{figure}

Now, we are in condition to compute the mean curvature of ruled surfaces and then proceed to characterize the helicoids as the only ruled minimal surfaces in Riemannian space forms. First, note that choosing $Z$ and $\alpha$ to be orthogonal, gives $g_{12}=0$ and, consequently, to find the mean curvature it is enough to compute $h_{11}=\langle\nabla_{\psi_u}\psi_u ,\xi\rangle$. Indeed,
\begin{equation}\label{eq::MeanCurvOfRuledSurface}
H=\frac{1}{2}\frac{h_{11}g_{22}-2h_{12}g_{12}+h_{22}g_{11}}{g_{11}g_{22}-g_{12}^2}=\frac{h_{11}}{2g_{11}}.
\end{equation}
In addition, we assume that the directrix curve $\alpha$ is parametrized by arc-length.   

Given a geodesic $\gamma$, we may consider the 1-parameter subgroup of isometries $\lambda_t$ obtained by the composition of the translation along $\gamma$ and the 1-parameter subgroup whose orbits are circles centered on points of $\gamma$. Up to isometries, in the hyperbolic space $\mathbb{H}^3(r)$ we may write for some constant $\omega\in\mathbb{R}$ \cite{RipollNMJ1989}
\begin{equation}
    u\mapsto \lambda_u=\left(
    \begin{array}{cccc}
        \cosh \left(\frac{u}{r}\right) & \sinh \left(\frac{u}{r}\right) & 0 & 0  \\ [4pt]
        \sinh \left(\frac{u}{r}\right) & \cosh \left(\frac{u}{r}\right) & 0 & 0  \\ [4pt]
        0 & 0 & \cos(\omega u) & -\sin(\omega u) \\
        0 & 0 & \sin(\omega u) &  \cos(\omega u)  \\
    \end{array}
    \right),
\end{equation}
while in the 3-sphere $\mathbb{S}^3(r)$ we may write
\begin{equation}
    u\mapsto \lambda_u=\left(
    \begin{array}{cccc}
        \cos \left(\frac{u}{r}\right) & -\sin \left(\frac{u}{r}\right) & 0 & 0  \\[4pt]
        \sin \left(\frac{u}{r}\right) & \cos \left(\frac{u}{r}\right) & 0 & 0 \\ [4pt]
        0 & 0 & \cos(\omega u) & -\sin(\omega u) \\
        0 & 0 & \sin(\omega u) &  \cos(\omega u)  \\
    \end{array}
    \right).
\end{equation}
We call $\lambda_u$ a \emph{helicoidal motion} with \emph{angular pitch} $\omega$ and \emph{axis} $\gamma$. A \emph{helicoid} is the surface obtained by applying a helicoidal motion to a geodesic $\beta$ in the totally geodesic surface orthogonal to the axis $\gamma$. If the axis is $\gamma:\pm x_1^2+x_2^2=\pm r^2$, as above, then we can take $\beta(v)=r(\cosh\frac{v}{r},0,\sinh\frac{v}{r},0)$ in $\mathbb{H}^3(r)$ and $\beta(v)=r(\cos\frac{v}{r},0,\sin\frac{v}{r},0)$ in $\mathbb{S}^3(r)$, which  implies that up to isometries any helicoid can be written in $\mathbb{H}^3(r)$ as [see Figure \ref{fig:Helicoids}, left]
\begin{eqnarray}
    \lambda_u(\beta(v)) &=& r\left(\cosh\frac{v}{r}\cosh\frac{u}{r},\cosh\frac{v}{r}\sinh\frac{u}{r},\sinh\frac{v}{r}\cos \omega u,\sinh\frac{v}{r}\sin \omega u\right)\nonumber\\
    &=& \cosh\frac{v}{r}\left(r\cosh\frac{u}{r},r\sinh\frac{u}{r},0,0\right)+r\sinh\frac{v}{r}\left(0,0,\cos \omega u,\sin \omega u\right)\nonumber\\
    \label{eq::H3helicoids}
\end{eqnarray}
and in $\mathbb{S}^3(r)$ as [see Figure \ref{fig:Helicoids}, right]
\begin{eqnarray}
\lambda_u(\beta(v)) &=& r\left(\cos\frac{v}{r}\cos\frac{u}{r},\cos\frac{v}{r}\sin\frac{u}{r},\sin\frac{v}{r}\cos \omega u,\sin\frac{v}{r}\sin \omega u\right)\nonumber\\
&=& \cos\frac{v}{r}\left(r\cos\frac{u}{r},r\sin\frac{u}{r},0,0\right)+r\sin\frac{v}{r}\left(0,0,\cos \omega u,\sin \omega u\right).\label{eq::S3helicoids}
\end{eqnarray}

We clearly see that these helicoids are ruled surfaces with generating curves $\alpha(u)=r(\cosh(u/r),\sinh(u/r),0,0)$ and $\alpha(u)=r(\cos(u/r),\sin(u/r),0,0)$ and rulings tangent to the vector field $Z(u)=(0,0,\cos(\omega u),\sin(\omega u))$. Moreover, as in Euclidean space, helicoids in space forms are also minimal surfaces. Indeed, we have that
$$DZ=Z'-\langle\alpha',Z\rangle\frac{\alpha}{r^2}=Z'=\omega(0,0,-\sin(\omega u),\cos(\omega u))$$ 
and $D^2Z=Z''=-\omega^2Z$. (We will do the calculations in $\mathbb{H}^3(r)$, the computations in $\mathbb{S}^3(r)$ being analogous.) Then, it follows that $(D^2Z,\alpha',Z)=0$ and $(D^2Z,DZ,Z)=0$. Since $\alpha$ is a geodesic, we also have $D\alpha'=0$. Putting all this together implies $h_{11}=0$ and since $\langle\alpha',Z\rangle=0$, we conclude that $H=0$. 

In the next theorem, we shall prove that every ruled minimal surface in a space form must be a helicoid \cite{doCarmoTAMS1983,LawsonAM1970}. For our proof, we follow similar steps as those in Theorem 3.1 of Ref. \cite{Barbosa1986} in Euclidean space. In order to do that, it should be first mentioned that as in Euclidean space, also in space forms a point moving under a helicoidal motion generate \emph{helices}, i.e., curves of constant curvature and torsion. Indeed, for example, the curve $u\mapsto \gamma_v(u)\equiv \lambda_u(\beta(v))$ in $\mathbb{H}^3(r)$ {has} curvature and torsion given by
\begin{equation}
    \kappa(u) = \frac{\sqrt{\langle \gamma_v'\times\gamma_v'',\gamma_v'\times\gamma_v''\rangle}}{\langle \gamma_v',\gamma_v'\rangle^{\frac{3}{2}}}=\frac{(1+r^2\omega^2)\sinh(2v/r)}{1-r^2\omega^2+(1+r^2\omega^2)\cosh(2v/r)}
\end{equation}
and
\begin{equation}
    \tau(u) = \frac{\langle \gamma_v'\times\gamma_v'',\gamma_v'''\rangle}{\langle \gamma_v'\times\gamma_v'',\gamma_v'\times\gamma_v''\rangle}=\frac{2\omega}{r[1-r^2\omega^2+(1+r^2\omega^2)\cosh(2v/r)]}.
\end{equation}
{Since we are modeling $\mathbb{H}^3(r)$ as a hypersphere in Lorentzian space, $\gamma_v'$ and $\gamma_v''$ in the expressions for $\kappa$ and $\tau$ can be calculated via ordinary differentiation instead of the covariant derivative in $\mathbb{H}^3(r)$.} Conversely, it can be shown that up to isometries of $\mathbb{H}^3(r)$ or $\mathbb{S}^3(r)$, every helix can be written as in Eqs. \eqref{eq::H3helicoids} or \eqref{eq::S3helicoids} for some $v$ and $\omega$ depending on the values of $\kappa$ and $\tau$. (See, for example, Theorem 1 of \cite{ChakrabartiInvolve2019}, whose proof can be easily adapted to helices in $\mathbb{H}^3(r)$.) 

\begin{theorem}\label{Thr::CharRuledMinSurf}
Let $\Sigma^2$ be a minimal ruled surface in $\mathbb{H}^3(r)$, or in $\mathbb{S}^3(r)$, then $\Sigma^2$ is a helicoid, i.e., up to isometries, $\Sigma^2$ is parametrized as in Eq. $(\ref{eq::H3helicoids})$, or in Eq. $(\ref{eq::S3helicoids})$, respectively.
\end{theorem}
\begin{proof}
From Eq. \eqref{eq::MeanCurvOfRuledSurface}, it follows that $H=0$ if, and only if, $h_{11}=0$. Equivalently, $\Sigma^2$ is a minimal ruled surface if, and only if, 
\begin{eqnarray}
    (D\alpha',\alpha',Z)&= & 0, \label{eqhel1} \\
     (D\alpha',DZ,Z)   +(D^2Z,\alpha',Z)&=&0,\label{eqhel2} \\
     (D^2Z,DZ,Z)&=&0,\label{eqhel3}
\end{eqnarray}
{where we have adopted the shorthand notation $D=\nabla_{\alpha'}$ and $D^2=\nabla_{\alpha'}\nabla_{\alpha'}$.}

From Eq. (\ref{eqhel1}), we have $D\alpha'\in \mbox{\mbox{\mbox{span}}}\{\alpha',Z\}$. Since $\Vert\alpha'\Vert=1$, what gives $\langle D\alpha',\alpha'\rangle=0$, then $D\alpha'$ and $Z$ are parallel. Thus, Eq. (\ref{eqhel2}) gives
$$
(D^2Z,\alpha',Z)=0.
$$ 
This relation, together with Eq. (\ref{eqhel3}), implies that $D^2Z\in \mbox{span}\{\alpha',Z\}$ and $D^2Z\in \mbox{span}\{DZ,Z\}$. On the one hand, if $D^2Z$ and $Z$ are not parallel, we necessarily have
$$
\mbox{span}\{\alpha',Z\}=\mbox{span}\{DZ,Z\}.
$$
In addition, since $\langle DZ,Z\rangle=0$ and $\langle\alpha',Z\rangle=0$, then $\alpha'$ must be parallel to $DZ$. Thus, using that $D\alpha'=\mu Z$ for some $\mu$, as inferred above, the torsion $\tau$ of $\alpha$ is
\begin{equation}
    \tau=\frac{(\alpha',D\alpha',D^2\alpha')}{\Vert D\alpha'\Vert^2}=(\alpha', Z,DZ)=0.
\end{equation}
Now, from Corollary \ref{Cor::KextRuledSurf}, the condition $\tau=0$ implies that, in addition to $H=0$, $\Sigma^2$ is also extrinsically flat. Consequently, it is part of a totally geodesic surface, which is a helicoid with zero angular pitch.

On the other hand, let us assume that $D^2Z$ and $Z$ are parallel. If $DZ$ and $\alpha'$ are also parallel, then in addition to $H=0$ it follows that $K=0$ and, therefore, $\Sigma^2$ would be a totally geodesic surface. Thus, we may assume that $DZ$ and $\alpha'$ are not parallel. Using that $D\alpha'$ and $Z$ are parallel, $D\alpha'\parallel Z$, the curvature function $\kappa$ of $\alpha$ is $\kappa=\pm\langle D\alpha',Z\rangle=\mp\langle\alpha',DZ\rangle$. Therefore, $Z\parallel D\alpha'$ together with $Z\parallel D^2Z$ finally give 
\begin{equation}
 \pm\,\kappa'=\langle D\alpha',DZ\rangle+\langle\alpha',D^2Z\rangle=0.
\end{equation}
Analogously, the derivative of the torsion of $\alpha$ is
\begin{equation}
    \pm\,\tau'=\langle D\alpha'\times Z,DZ\rangle+\langle\alpha'\times DZ,DZ\rangle+\langle\alpha'\times Z,D^2Z\rangle=0.
\end{equation}
In short, $\kappa$ are $\tau$ are constant, what tells us that $\alpha$ is a helix and, therefore, there exists a value of $v$ such that we can write $\alpha$ as (the proof for $\mathbb{S}^3(r)$ is entirely analogous)
\begin{equation}
    \alpha(u) = r(\cosh\frac{v}{r}\cosh\frac{u}{r},\cosh\frac{v}{r}\sinh\frac{u}{r},\sinh\frac{v}{r}\cos(\omega u),\sinh\frac{v}{r}\sin(\omega u)).
\end{equation}
The unit tangent is $\mathbf{T}=\Omega\,\alpha'$, where $\Omega^{-1}=\sqrt{\cosh^2(v/r)+r^2\omega^2\sinh^2(v/r)}$ does not depend on $u$. Since   $\kappa\mathbf{N}=\nabla_{\mathbf{T}}\mathbf{T}=\Omega^2D\alpha'$, where $\mathbf{N}$ is the principal normal, we conclude that $\mathbf{N}$ is a multiple of $D\alpha'=\alpha''-\langle\alpha',\alpha'\rangle\frac{\alpha}{r^2}=\alpha''-\frac{\alpha}{r^2\Omega^2}$. In addition, using that $Z\parallel D\alpha'$, we conclude that we may take $Z=\pm\mathbf{N}$. Thus 
\begin{equation}
    Z=-(\sinh\frac{v}{r}\cosh\frac{u}{r},\sinh\frac{v}{r}\sinh\frac{u}{r},\cosh\frac{v}{r}\cos(\omega u),\cosh\frac{v}{r}\sin(\omega u))
\end{equation}
and we can parametrize $\Sigma^2$ as $\psi(u,w) = \cosh\frac{w}{r}\,\alpha(u)+r\sinh\frac{w}{r}Z(u)$, which gives
\[
     \psi
     = r(\mbox{ch}\frac{v-w}{r}\mbox{ch}\frac{u}{r},\mbox{ch}\frac{v-w}{r}\mbox{sh}\frac{u}{r},\mbox{sh}\frac{v-w}{r}\cos\omega u,\mbox{sh}\frac{v-w}{r}\sin\omega u),
\]
where $\mbox{ch}=\cosh$ and $\mbox{sh}=\sinh$. Finally, comparison with Eq. \eqref{eq::H3helicoids} shows that $\psi(u,w)=\lambda_u(\beta(v-w))$, from which we conclude that $\Sigma^2$ is a helicoid.
\qed \end{proof}

\subsection{Singular points of ruled surfaces}

When we allow ruled surfaces in $\mathbb{E}^3$ to have singular points, the notion of striction curve plays an important role, since the singular points (if they exist) have to be contained in the image of the striction curve \cite{DoCarmo2016DG}. (As the example {of a cone} shows, the striction curve may degenerate to a single point.) The next proposition establishes the existence of striction curves for ruled surfaces in $\mathbb{S}^3(r)$ and $\mathbb{H}^3(r)$. Later, assuming that the directrix curve is the striction curve leads to a simpler description of the behavior of the Gaussian curvature, which will prove useful in classifying extrinsically flat surfaces.

\begin{proposition}\label{Prop::ExistStrictionCurve}
Let $\Sigma^2:(u,v)\mapsto\exp_{\alpha(u)}(vZ(u))$ be a ruled surface in $\mathbb{S}^3(r)$, or $\mathbb{H}^3(r)$, then there exists a curve $\beta:I\to\Sigma^2$, called the \emph{striction curve}, such that
\begin{equation}
    \langle\beta',\nabla_{\beta'}PZ\rangle = 0,
\end{equation}
where $PZ$ is a unit vector field along $\beta$ obtained by parallel transporting $Z$ along the geodesics connecting $\alpha$ to $\beta$. Up to antipodal points, the striction curve is unique. In addition, the (oriented) distance, $v(u)$, from $\beta$ to $\alpha$ is given in $\mathbb{S}^3(r)$ by
\begin{equation}\label{estricao2}
v(u)=\frac{r}{2}\arctan\left(-\frac{2}{r}\dfrac{\langle\alpha',\nabla_{\alpha'}Z\rangle}{\langle \nabla_{\alpha'}Z,\nabla_{\alpha'}Z\rangle-\frac{g(u,0)}{r^2}}\right),
\end{equation}
where $g=\det g_{ij}$, and in $\mathbb{H}^3(r)$ by
\begin{equation}
    v(u)
    =
    \frac{r}{2}\,\mathrm{arctanh}\left(-\frac{2}{r}\dfrac{\langle\alpha',\nabla_{\alpha'}Z\rangle}{\langle \nabla_{\alpha'}Z,\nabla_{\alpha'}Z\rangle 
    +
    \frac{g(u,0)}{r^2}}\right).
\end{equation}
\end{proposition}
\begin{proof}
Any curve  $\beta$ in {$\Sigma^2\subset\mathbb{S}^3(r)$} may be written  as  
\begin{equation}\label{estricao}
    \beta(u)=\text{exp}_{\alpha(u)}(v(u)Z(u))=\cos{\left(\dfrac{v(u)}{r}\right)}\alpha(u)+r\sin{\left(\dfrac{v(u)}{r}\right)}Z(u)
\end{equation}
for some function  $v:I\rightarrow\mathbb{R}$. The corresponding velocity vector field is
\begin{equation}
     \beta' = -\dfrac{v'}{r}s_v\alpha+ v'c_vZ+c_v\alpha'+rs_vZ' = v'PZ+c_v\alpha'+rs_vZ'.
\end{equation}
Now, using that $\langle PZ,\nabla_{\beta'}PZ\rangle=0$, we find
\begin{equation}\label{estricao1}
 \langle\beta' ,\nabla_{\beta'}PZ\rangle = c_v\langle\alpha',\nabla_{\beta'}PZ\rangle
+rs_v\langle Z',\nabla_{\beta'}PZ\rangle .
\end{equation}
The idea now is to impose $\langle\beta' ,\nabla_{\beta'}PZ\rangle=0$ and look for necessary conditions for the existence of a striction curve. However, we must  write $\langle\alpha',\nabla_{\beta'}PZ\rangle$ and $\langle Z',\nabla_{\beta'}PZ\rangle$ in terms of quantities that can be computed along $\alpha$ only. First, note that
\begin{eqnarray}
    \nabla_{\beta'}PZ & = & \frac{\partial}{\partial u}PZ\vert_{\beta}+\frac{\langle\beta',PZ\rangle}{r^2}\beta\nonumber\\
    & = & -\frac{s_v}{r}\alpha'+c_vZ'-\frac{v'}{r^2}\beta+\frac{\langle v'PZ+ c_v\alpha'+rs_vZ',-\frac{1}{r}s_v\alpha+c_vZ\rangle}{r^2}\beta\nonumber\\
    & = & -\frac{s_v}{r}\alpha'+c_vZ'+\frac{\langle\alpha',Z\rangle}{r^2}\beta.
\end{eqnarray}
Therefore, we have
\begin{eqnarray}
    \langle\alpha',\nabla_{\beta'}PZ\rangle & = & \langle\alpha',-\frac{s_v}{r}\alpha'+c_vZ'\rangle+\frac{\langle\alpha',Z\rangle}{r^2}\langle\alpha',c_v\alpha+rs_vZ\rangle\nonumber\\
    & = & -\frac{s_v}{r}+c_v\langle\alpha',\nabla_{\alpha'}Z\rangle+\frac{s_v}{r}\langle\alpha',Z\rangle^2,
\end{eqnarray}
where in the last equality we used that $\langle\alpha',Z'\rangle=\langle\alpha',\nabla_{\alpha'}Z\rangle$. On the other hand,
\begin{eqnarray}
    \langle Z',\nabla_{\beta'}PZ\rangle & = & \langle Z',-\frac{s_v}{r}\alpha'+c_vZ'\rangle+\frac{\langle\alpha',Z\rangle}{r^2}\langle Z',c_v\alpha+rs_vZ\rangle\nonumber\\
     & = & -\frac{s_v}{r}\langle \alpha',\nabla_{\alpha'}Z\rangle+c_v\langle Z',Z'\rangle-c_v\frac{\langle\alpha',Z\rangle^2}{r^2}\nonumber\\
     & = & -\frac{s_v}{r}\langle \alpha',\nabla_{\alpha'}Z\rangle+c_v\langle \nabla_{\alpha'}Z,\nabla_{\alpha'}Z\rangle,
\end{eqnarray}
where in the last equality we used that $\langle\nabla_{\alpha'}Z,\nabla_{\alpha'}Z\rangle=\langle Z',Z'\rangle-\frac{1}{r^2}\langle\alpha',Z\rangle^2$.

Thus, using that $\langle\alpha',\nabla_{\beta'}PZ\rangle$ and $\langle Z',\nabla_{\beta'}PZ\rangle$ are given in terms of quantities along $\alpha$, we can write
\begin{eqnarray*}
    \langle\beta' ,\nabla_{\beta'}PZ\rangle & = & -\frac{s_vc_v}{r}+c_v^2\langle\alpha',\nabla_{\alpha'}Z\rangle+\frac{s_vc_v}{r}\langle\alpha',Z\rangle^2\nonumber\\
    & - & s_v^2\langle \alpha',\nabla_{\alpha'}Z\rangle+rs_vc_v\langle \nabla_{\alpha'}Z,\nabla_{\alpha'}Z\rangle\nonumber\\
    & = &  -\Big(1-\langle\alpha',Z\rangle^2\Big)\frac{s_{2v}}{2r}+c_{2v}\langle\alpha',\nabla_{\alpha'}Z\rangle+\frac{rs_{2v}}{2}\langle \nabla_{\alpha'}Z,\nabla_{\alpha'}Z\rangle.
\end{eqnarray*}
Finally, by imposing that $\beta$ is the striction curve, i.e., $\langle\beta',\nabla_{\beta'}PZ\rangle=0$, and noting that $g(u,0)=\det g_{ij}(u,0)=\langle\alpha',\alpha'\rangle\langle Z,Z\rangle-\langle\alpha',Z\rangle^2=1-\langle\alpha',Z\rangle^2$, we conclude that $v=v(u)$ should satisfy
\[
v(u)=\frac{r}{2}\arctan\left(-\frac{2}{r}\dfrac{\langle\alpha',\nabla_{\alpha'}Z\rangle}{\langle \nabla_{\alpha'}Z,\nabla_{\alpha'}Z\rangle-\frac{g(u,0)}{r^2}}\right).
\]
In conclusion, defining $\beta$ as in Eq. (\ref{estricao}) with $v=v(u)$ given by Eq. (\ref{estricao2}), we obtain the striction curve of our ruled surface. 

For ruled surfaces in $\mathbb{H}^3(r)$, the solution for $v=v(u)$ is given by
\[
    v(u)
    =
    \frac{r}{2}\,\mbox{arctanh}\left(-\frac{2}{r}\dfrac{\langle\alpha',\nabla_{\alpha'}Z\rangle}{\langle \nabla_{\alpha'}Z,\nabla_{\alpha'}Z\rangle 
    +
    \frac{g(u,0)}{r^2}}\right).
\]

It remains to show the uniqueness of the striction curve. (On the 3-sphere, uniqueness should be understood up to antipodals, i.e., the antipodal image of $\beta$ it is also a striction curve.) First, note the importance of working with a variable radius. It allows us to check consistence of our answers with known Euclidean results: e.g., in the limit $r\gg1$, the expressions above lead to $v(u)\approx-\langle\alpha',Z'\rangle/\langle Z',Z'\rangle$, which is thus compatible with the Euclidean result \cite{DoCarmo2016DG}. Now, assume that there exists a ruled surface on the 3-sphere, or hyperbolic space, with two striction curves $\beta$ and $\bar{\beta}$, with $\bar{\beta}\not=-\beta$ in the case of $\mathbb{S}^3(r)$. In the limit $r\gg1$, this would lead to an Euclidean ruled surface with two striction curves, which contradicts the known uniqueness of striction curves in $\mathbb{E}^3$. (In the limit $r\gg1$ the antipodal striction curve goes to infinity.)
\qed \end{proof}


\begin{figure}[t]
    \centering
    \includegraphics[width=0.75\linewidth]{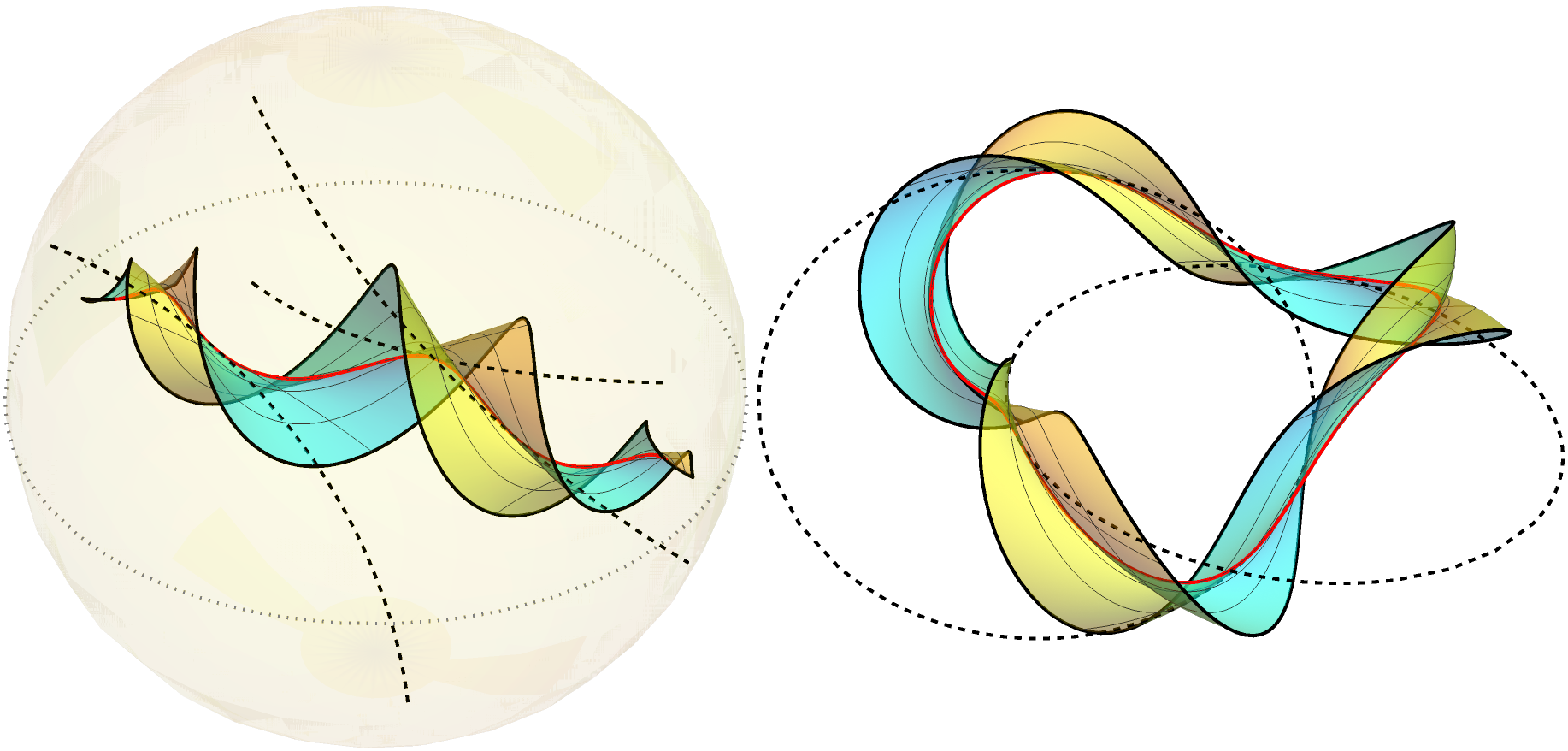}
    \caption{Tangent surface of a helix in space forms: (Left) The tangent surface in the Poincar\'e ball model of $\mathbb{H}^3(1)$ obtained from Eq. \eqref{eq::H3helicoids} through stereographic projection $(x_1,x_2,x_3,x_4)\mapsto (\frac{x_2}{1+x_1},\frac{x_3}{1+x_1},\frac{x_4}{1+x_1})$. The dotted gray line indicates the points $(x,y,z)$ in the ideal boundary with $z=0$, while the dashed black curves correspond to geodesics tangent to the generating curve; (Right) The tangent surface in $\mathbb{S}^3(1)$ obtained from Eq. \eqref{eq::S3helicoids} through stereographic projection $(x_1,x_2,x_3,x_4)\mapsto (\frac{x_1}{1-x_4},\frac{x_2}{1-x_4},\frac{x_3}{1-x_4})$. The dashed black curves correspond to geodesics tangent to the generating curve. (In the figures, the full red lines are the generating curves, which are helices corresponding to $\omega=3.0$ and $v=0.2$ in Eqs. \eqref{eq::H3helicoids} and \eqref{eq::S3helicoids}. Both surfaces have constant width. Note that in $\mathbb{S}^3$ helices are periodic if and only if $\omega\in\mathbb{Q}$ \cite{ChakrabartiInvolve2019}.) {Figures generated with Mathematica.}}
    \label{fig:TangSurf}
\end{figure}

\subsection{Extrinsically flat ruled surfaces}

Corollary \ref{Cor::KextRuledSurf} provides a necessary and sufficient condition for the extrinsic Gaussian curvature to vanish. In this subsection, we shall further simplify the expression of $K_{ext}$ with the help of the striction curve and the distribution parameter to be defined below.

An interesting example of extrinsically flat ruled surfaces is given by the tangent surfaces. Choosing a regular curve $\alpha(u)$, its \emph{tangent surface} is the ruled surface with rulings tangent to $Z(u)=\alpha'(u)$. Consequently, we have $(\alpha',Z,\nabla_{\alpha'}Z)=0$ along all points of $\alpha$. Therefore, from Corollary \ref{Cor::KextRuledSurf} every tangent surface has vanishing extrinsic curvature on its regular points. (See Figure \ref{fig:TangSurf} for examples of tangent surfaces.) It is straightforward to check that every tangent surface is singular along their generating curve. This poses the problem of locating the singular points of generic ruled surfaces in space forms. As in Euclidean space, this can be accomplished by using the striction curve.  

From now on, we may assume that the generating curve $\alpha$ is the striction curve of $\Sigma^2$, Prop. \ref{Prop::ExistStrictionCurve}. From Eq. \eqref{eq::TangentVecPsi}, we have
\begin{eqnarray}\label{estricao4}
        \psi_u\times\psi_v & =&
        \cos{\left(\dfrac{v}{r}\right)}\alpha'\times PZ\times\psi+r\sin{\left(\dfrac{v}{r}\right)}\nabla_{\alpha'}Z\times PZ\times\psi\nonumber\\
        &=& \sec\left(\frac{v}{r}\right)\left[\cos{\left(\dfrac{v}{r}\right)}\alpha'\times Z\times\psi+r\sin{\left(\dfrac{v}{r}\right)}\nabla_{\alpha'}Z\times Z\times\psi\right],
\end{eqnarray}
where in the last equality we used a similar process as the one leading to Eq. (\ref{Eq::fromAWDwToAwDw}) to replace $PZ$ by $Z$. Since $\langle\alpha',\nabla_{\alpha'}Z\rangle=0$ (striction curve) and $\langle\nabla_{\alpha'}Z,Z\rangle=0$ ($\Vert Z\Vert=1$), we have $\alpha'\times Z\times\psi=\lambda \nabla_{\alpha'}Z$. Thus, Eq. (\ref{estricao4}) becomes
\begin{eqnarray}
   \psi_u\times\psi_v &=& \sec\left(\frac{v}{r}\right)\left[\lambda\cos{\left(\dfrac{v}{r}\right)}\nabla_{\alpha'}Z+r\sin{\left(\dfrac{v}{r}\right)}\nabla_{\alpha'}Z\times Z\right].
\end{eqnarray}
From $g=\det (g_{ij})=\Vert\psi_u\times\psi_v\Vert^2$, the equation above leads to
\begin{eqnarray}\label{estricao6}
        g &=&\sec^{2}{\left(\dfrac{v}{r}\right)}\left[\lambda^2\cos^2{\left(\dfrac{v}{r}\right)}\Vert\nabla_{\alpha'}Z\Vert^2 +r^2\sin^2{\left(\dfrac{v}{r}\right)}\,\Vert\nabla_{\alpha'}Z\times Z\times\psi\Vert^2\right]\nonumber\\
        &=& \sec^{2}{\left(\dfrac{v}{r}\right)}\left[\lambda^2\cos^2{\left(\dfrac{v}{r}\right)}+r^2\sin^2{\left(\dfrac{v}{r}\right)}\right]\, \Vert\nabla_{\alpha'}Z\Vert^2.\label{eq::detGUsingLambda}
\end{eqnarray}  

We shall call $\lambda$ the \emph{distribution parameter} of the ruled surface $\Sigma^2$, which can be written as
\begin{equation}\label{eq::DefDistParameter}
\lambda\,\Vert\nabla_{\alpha'}Z\Vert^2=(\nabla_{\alpha'}Z,\alpha',Z)\Rightarrow \lambda=\frac{(\alpha',Z,\nabla_{\alpha'}Z)}{\Vert\nabla_{\alpha'}Z\Vert^2}.    
\end{equation}

Note that the existence of singular points of a ruled surface $\Sigma^2$, i.e., points where $\Vert\psi_u\times\psi_v\Vert=0$, is equivalent to having  $\sin(\frac{v}{r})=0$ and $\lambda=0$ simultaneously. Then, the locus of singular points of   $\Sigma^2\subset\mathbb{S}^3(r)$ should be contained in the striction curve or on its antipodals points, if they exist. (This is not an unexpected result at all, if something fails at $p\in\mathbb{S}^3(r)$, usually it also fails at the antipodal $-p$.) This relation between the singularities of a ruled surface and its corresponding striction curve is completely analogous to what happens in Euclidean space \cite{DoCarmo2016DG}. {A similar conclusion is valid in $\mathbb{H}^3(r)$. Note that the definition of the distribution parameter, Eq. \eqref{eq::DefDistParameter}, is the same in $\mathbb{S}^3(r)$, $\mathbb{E}^3$, and $\mathbb{H}^3(r)$.}

\begin{remark}
The singularities of ruled surfaces in  $\mathbb{S}^3(r)$ have been investigated in  \cite{IzumiyaDGA2011}. In fact,  $c_{\kappa}(t)=c_1(t)c_6(t)+c_3(t)c_4(t)$ can be related to the distribution parameter (compare Prop. \ref{Prop::KextRuledSurfDistPar} below with Prop. 4.2. of \cite{IzumiyaDGA2011}), while  the solutions $\theta_0$ of Eq. (4.5) of \cite{IzumiyaDGA2011} parametrizes the striction curve. Note, however, the expressions in   \cite{IzumiyaDGA2011} for the distribution parameter, striction curve, and the corresponding extrinsic Gaussian curvature are model-dependent since the coefficients $c_i(t)$  make sense individually only when $\mathbb{S}^3(r)$ is seen as a hypersurface in $\mathbb{E}^4$. 
\end{remark} 

Note that $K_{ext}$ vanishes if, and only if, the distribution parameter vanishes, see Corollary \ref{Cor::KextRuledSurf}. In addition, by taking the striction curve as the generating curve and using the distribution parameter, we can write $K_{ext}$  as
\begin{proposition}\label{Prop::KextRuledSurfDistPar}
Let $\Sigma^2\subset\mathbb{S}^3(r)$, {or $\Sigma^2\subset\mathbb{H}^3(r)$,} be a ruled surface with the rulings tangent to $Z$ and whose directrix $\alpha$ is the striction curve. Then, the extrinsic curvature of $\Sigma^2$ is 
\begin{equation}\label{eq::KextUsingDistParam}
{    K_{ext}(u,v)=-c_v^2\,\dfrac{\lambda(u)^2}{\left[\lambda(u)^2c_v^2+r^2s_v^2\right]^2}.    }
\end{equation}
where $\lambda$ is the distribution parameter defined in Eq. \eqref{eq::DefDistParameter},  {$s_v=\sin(\frac{v}{r})$ and $c_v=\cos(\frac{v}{r})$ in $\mathbb{S}^3(r)$, and $s_v=\sinh(\frac{v}{r})$ and $c_v=\cosh(\frac{v}{r})$ in $\mathbb{H}^3(r)$,}
\end{proposition}

\begin{remark}
In the limit $r\gg1$, the sphere $\mathbb{S}^3(r)$ {and the hyperbolic space $\mathbb{H}^3(r)$ are} locally Euclidean. Up to second order in $\frac{1}{r}\ll1$, we have $c_v\approx1$ and $s_v\approx \frac{v}{r}$. Then, for $r\gg1$, the extrinsic Gaussian curvature of a {space form} ruled surface  is  $K_{ext}\approx-\frac{\lambda^2}{[\lambda^2+v^2]^2}$, which agrees with the Euclidean result \cite{DoCarmo2016DG}.
\end{remark}

It is known that in Euclidean  \cite{TuncerGMN2015} and hyperbolic \cite{EtayoTJM2017} spaces there is a close relationship between the notion of rotation minimizing frames and developable surfaces. (A unit vector field $Z$ normal to a curve $\alpha:I\to M^3$ is said to minimize rotation along $\alpha$ if $\nabla_{\alpha'}Z$ and $\alpha'$ are parallel \cite{BishopMonthly,Etayo2016,dSdS20}.) The same relation is also valid in $\mathbb{S}^3(r)$ {and in $\mathbb{H}^3(r)$} and it follows as a corollary of Proposition \ref{Prop::KextRuledSurfDistPar}:

\begin{proposition}
The extrinsic curvature of a ruled surface {$\Sigma^2$ in $\mathbb{S}^3(r)$, or in $\mathbb{H}^3(r)$,} vanishes if and only if $\alpha'$, $Z$, and $\nabla_{\alpha'}Z$ are linearly dependent. In particular, if $Z$ is a vector field normal to $\alpha$, then $K_{ext}=0$ if, and only if, $Z$ minimizes rotation along $\alpha$.
\end{proposition}

Let us now proceed with a characterization of developable surfaces. 
We have to distinguish between two cases
\begin{enumerate}
    \item If $\nabla_{\alpha'}Z\times Z=0$, then we should have $\nabla_{\alpha'}Z=0$  since $Z$ is a unit vector field. In other words, $\Sigma^2$ is a cylindrical surface.
    \item Assume $\nabla_{\alpha'}Z\times Z\not=0$ in all points of $\alpha$, which in particular implies $\nabla_{\alpha'}Z\not=0$. From Eq. (\ref{eq::KextUsingDistParam}), we should study surfaces for which
\begin{equation}
\lambda=0 \Leftrightarrow (\alpha',Z,\nabla_{\alpha'}Z)=0.    
\end{equation}
   If follows from Eq. (\ref{eq::detGUsingLambda}) that the striction curve $\alpha$ [and its antipodal $-\alpha$, when $\Sigma^2\subset\mathbb{S}^3(r)$] is the locus of the singular points of $\Sigma^2$. If $\alpha'\equiv0$, then the striction curve is a point and the corresponding ruled surface can be seen as a cone with one vertex when $\Sigma^2\subset\mathbb{H}^3(r)$ and two vertexes when $\Sigma^2\subset\mathbb{S}^3(r)$.  On the other hand, if $\alpha$ is regular in all its points, i.e., $\alpha'\not=0$, then from $\langle \alpha',\nabla_{\alpha'}Z\rangle=0$ we should have $Z$ and $\alpha'$ parallel, which then says that $\Sigma^2$ is the tangent surface of $\alpha$.
\end{enumerate}

The characterization above is analogous to the Euclidean one \cite{DoCarmo2016DG}. This happens because we were able to introduce a distribution parameter. Then, dividing the domain of $\alpha$ in open intervals where each of the above conditions hold, every developable surface is a union of pieces of cylinders, cones, and tangent surfaces. (If the zeros of $\nabla_{\alpha'}Z\times Z$ and $\alpha'$ form a discrete set, then the set of points where the characterization above breaks is a union of curves in $\Sigma^2$. The case where there exist accumulation points is not discussed here.)

\section{Characterization of extrinsically flat surfaces and space forms}

In this section, we show that extrinsically flat surfaces in space forms must be ruled. Though not new, to the best of our knowledge all the known proofs are model dependent \cite{HondaTMJ2012,IzumiyaDGA2011,PortnoyPJM1975,Spivak1979v4}. We shall provide a new and model-independent proof for this classification, which has the advantage of allowing us to identify the essential property that makes the {implication} works. We then prove that space forms can be characterized by the property that there exist  extrinsically flat surfaces tangent to any given 2-plane and that they all must be ruled. Therefore, our theorem can be seen as a generalization of the characterization of space forms by the existence of totally geodesic surfaces tangent to any given plane at any point (Cartan theorem \cite{Cartan1946,Spivak1979v4}), i.e., totally geodesic can be replaced by extrinsically flat and ruled.

\begin{definition}
A smooth surface $\Sigma^2$ of a Riemannian manifold $M^3$ is said to be \emph{developable} if {there exists a coordinate system $(u,v)$ where the} Gauss normal map $\xi$ is stationary along one of the family of coordinate curves, {e.g., $\nabla_{\partial_v}\,\xi=0$}. 
\end{definition}

\begin{lemma}\label{lemma::DevIFFKextZero}
If a surface  $\Sigma^2\subset M^3$ is developable, then its extrinsic Gaussian curvature $K_{ext}$ vanishes. Reciprocally, if $K_{ext}$ vanishes and if $p_0\in\Sigma^2$ is a parabolic point\footnote{{A point $p$ is said to be a \emph{parabolic point} if only one of the principal curvatures at $p$ vanishes.}} or an interior point of the set of planar points\footnote{{A point $p$ is said to be a \emph{planar point} if both principal curvatures at $p$ vanish.}}, then $\Sigma^2$ is a developable surface in an open neighborhood of $p_0$.
\end{lemma}
\begin{proof}
If $\Sigma^2$ is developable, then its spherical image under $\xi$ is a curve and, consequently, $K_{ext}=0$.

Reciprocally, if $K_{ext}\equiv 0$, then  $\kappa _1\kappa _2=0$ and we may assume, without loss of generality, that $\kappa_2\equiv 0$ in an open neighborhood $\mathcal{U}$ of $p_0$. Let   $V\in T_p\Sigma$, $p\in\Sigma$, be a unit vector such that $\kappa _n(p;V)=\kappa_2(p)=0$. Thus, $V$ is both an asymptotic and a principal direction and it follows that
\begin{equation}
\nabla_V\xi=-\kappa_n V = 0.
\end{equation}
Parametrizing $\Sigma^2$ with a coordinate system $(u,v)$ in the neighborhood $\mathcal{U}$ in which the $v$-curves are tangent to $V$ and the $u$-curves are tangent to the direction perpendicular to $V$ shows that $\xi$ is stationary along one of the family coordinate curves: $\exists\,A,\,\partial_v=AV\Rightarrow \nabla_{\partial_v}\xi=A\nabla_{V}\xi=0$.
\qed \end{proof}

In Euclidean space, surfaces with vanishing (extrinsic) Gaussian curvature should be ruled. It is known that the same is also true in space forms. Now, we shall provide a model-independent proof for this implication.

\begin{theorem}\label{thr::ExtFlatImpliesRuledInSpcForm}
Let $\Sigma^2$ be a surface with $K_{ext}=0$ in a space form $M^3$. If $p_0\in\Sigma^2$ is a parabolic point or an interior point of the set of planar points, then $\Sigma^2$ is a ruled surface in an open neighborhood of $p_0$.
\end{theorem}
\begin{proof}
Let $P$ be the set of planar points of $\Sigma^2$. If $p_0\in \mathcal{U}\subset P$, $\mathcal{U}$ open, then $\xi$ is constant in $\mathcal{U}$ and, therefore, it is a piece of a totally geodesic surface of $M$. In particular, $\mathcal{U}\subset\Sigma^2$ is a piece of a ruled surface. Indeed, take a curve $\alpha:I\to \mathcal{U}\subset\Sigma^2$ and let $\alpha_u(v)$ be a geodesic orthogonal to $\alpha$ at $\alpha(u)$. Since $\Sigma^2$ is totally geodesic, every $\alpha_u(v)$ is also a geodesic of $M^3$. Thus, $\Sigma^2$ is ruled.

On the other hand, if $p_0\not\in P$, then by continuity there exists a neighborhood $\mathcal{U}$ of $p_0$ containing only parabolic points of $\Sigma^2$. Let $\{\partial_1,\partial_2,\partial_3\}$ be a {frame field on} $M^3$ such that $\partial_1\vert_{\Sigma^2},\partial_2\vert_{\Sigma^2}$ are the coordinate vector fields pointing in the principal directions in $\mathcal{U}$, with $\kappa_2=0$, and $\partial_3\vert_{\Sigma^2}=\xi$ is the unit normal of $\Sigma^2$. Now, if $R$ is the Riemann curvature tensor of $M^3$ and if we use Latin letters $i,j,k,\dots$ running from $1$ to $2$, we have for all $p\in \mathcal{U}\subseteq\Sigma^2$
\begin{eqnarray}
    \langle R(\partial_i,\partial_j)\partial_k,\xi\rangle_p & = & \langle\nabla_{\partial_j}\nabla_{\partial_i}\partial_k-\nabla_{\partial_i}\nabla_{\partial_j}\partial_k,\xi\rangle_p\nonumber\\
    &=& \langle\nabla_{\partial_j}(\Gamma_{ik}^m\partial_m+h_{ik}\xi)-\nabla_{\partial_i}(\Gamma_{jk}^m\partial_m+h_{jk}\xi),\xi\rangle_p\nonumber\\
    & = & \langle A \partial_1+B\partial_2+(h_{jm}\Gamma_{ik}^m-h_{im}\Gamma_{jk}^m+h_{ik,j}-h_{jk,i})\xi,\xi\rangle_p\nonumber\\
    R_{ijk3}(p) & = & h_{jm}\Gamma_{ik}^m-h_{im}\Gamma_{jk}^m+h_{ik,j}-h_{jk,i}\,,\label{eq::Rijk3}
\end{eqnarray}
where we used that $\xi(p)=\partial_3(p)$ and $A\partial_1$, $B\partial_2$ do not contribute to $R_{ijk3}$. In addition, from Lemma \ref{lemma::DevIFFKextZero}, we have
\begin{equation}\label{eq::h11hi2ForDevSurf}
h_{2i}=\langle \nabla_{\partial_2}\partial_i,\xi\rangle=-\langle \partial_i,\nabla_{\partial_2}\xi\rangle=0\Rightarrow h_{21}=h_{22}=0.
\end{equation}
Now, since $M^3$ is a space form and $i,j,k$ are tangent, i.e., $i,j,k\not=3$, we have $R_{ijk3}=K_0(\langle\partial_i,\partial_k\rangle\langle\partial_j,\partial_3\rangle-\langle\partial_i,\partial_3\rangle\langle\partial_j,\partial_k\rangle)=0$. Finally, applying this together with Eqs. (\ref{eq::Rijk3}) and (\ref{eq::h11hi2ForDevSurf}) for $j=k=2$, gives 
$$0=R_{i223}=h_{2r}\Gamma_{i2}^r-h_{ir}\Gamma_{22}^r+h_{i2,2}-h_{22,i}=-h_{i1}\Gamma_{22}^1.$$
We must have $h_{11}\not=0$ in $\mathcal{U}$, otherwise $\mathcal{U}$ would contain a planar point. Thus,  $\Gamma_{22}^1=0$ and it follows that
\[
  \nabla_{\partial_2}\partial_2=\Gamma_{22}^i\partial_i+h_{22}\xi=\Gamma_{22}^2\partial_2\Rightarrow\kappa_{g}(\partial_2)=\frac{1}{\sqrt{g_{11}}}(\frac{\rmd v}{\rmd s_v})^{2}\langle\nabla_{\partial_2}\partial_2,-\partial_1\rangle=0, 
\]
where $s_v$ is the arc-length parameter of the curves tangent to $\partial_2$. Therefore, the asymptotic lines in $\mathcal{U}\subset\Sigma^2$, i.e., curves tangent to $\partial_2$, are geodesics of $M^3$. In other words, $\mathcal{U}\subset\Sigma^2$ is a piece of a ruled surface.
\qed \end{proof}

At first, one may naively expect that $K_{ext}=0$ should imply the surface is ruled. After all, extrinsically flat surfaces are foliated by a one-parameter family of curves along which the normal is stationary (Lemma \ref{lemma::DevIFFKextZero} is valid on any manifold). 
However, a close look at the proof of Theorem \ref{thr::ExtFlatImpliesRuledInSpcForm} reveals that some sort of restriction on the curvature tensor is expected. Indeed, we  used in Theorem \ref{thr::ExtFlatImpliesRuledInSpcForm} that $R_{1223}=0$ is a sufficient condition for the validity of the implication ``$K_{ext}=0\Rightarrow \Sigma^2\mbox{ is ruled}$". The result below shows that $R_{1223}=0$ is also a necessary condition.

\begin{proposition}\label{Prop::R1223zeroIFFruled}
Let $\Sigma^2$ be an extrinsically flat  with unit normal $\xi$ in a generic manifold $M^3$. Let $(u,v)$ be a coordinate system and  $\partial_1$ and $\partial_2$ the corresponding tangent vector fields, where the $v$-curves have vanishing principal curvature. Then, the $v$-curves are geodesics, i.e., $\Sigma^2$ is ruled if, and only if,  $R_{1223}=0$, where $\partial_3=\xi$.      
\end{proposition}
\begin{proof}
We already know $R_{1223}=0$ is sufficient. Now, parametrize $\Sigma^2$ as $X(u,v)=\exp_{\alpha(u)}(v\,Z(u))$ for some unit vector field $Z$ along $\alpha$. Then, $\nabla_{\partial_2}\partial_2=0$ and $h_{22}=\langle\nabla_{\partial_2}\partial_2,\xi\rangle=0$. Since $K_{ext}=\frac{h}{g}=-\frac{(h_{12})^2}{g}$, where $h=\det(h_{ij})$ and $g=\det(g_{ij})$, from $K_{ext}=0$, we have $h_{12}=0$. Thus, the shape operator is $A=\mbox{diag}(h_{11},0)$ and, therefore, the coordinate curves  are lines of curvature. Repeating the reasoning on the proof of the previous theorem, it follows that $ R_{ijk3} = h_{jm}\Gamma_{ik}^m-h_{im}\Gamma_{jk}^m+h_{ik,j}-h_{jk,i}$. If all $h_{ij}$ vanish, then trivially $R_{1223}=0$. On the other hand, if $h_{11}\not=0$, then the geodesic curvature of the asymptotic lines $v\mapsto X(u,v)$ is
\begin{equation}
    \kappa_{g}(\partial_2)=\frac{1}{\sqrt{g_{11}}}(\frac{\rmd v}{\rmd s_v})^{2}\langle\nabla_{\partial_2}\partial_2,-\partial_1\rangle \propto \Gamma_{22}^1=-\frac{\langle R(\partial_1,\partial_2)\partial_2,\xi\rangle}{h_{11}},
\end{equation}
where $s_v$ is the arc-length parameter of the asymptotic lines. Finally, the assumption $\kappa_g(\partial_2)=0$ implies that $\langle R(\partial_1,\partial_2)\partial_2,\xi\rangle=0$, as desired.
\qed \end{proof}

Now, we are able to establish the main result of this work.

\begin{theorem}\label{Theor::CharSpcFraRuledSurf}
Let $M^3$ be a connected Riemannian manifold having the property that there exists an extrinsically flat surface tangent to any given 2d plane and also that every extrinsically flat surface in $ M^3$ is ruled. Then, $M^3$ must have constant sectional curvature.
\end{theorem}
\begin{proof}
Let $\pi_p$ be a 2-plane of $T_pM^3$ and let $\Sigma^2$ be an extrinsically flat surface such that $T_p\Sigma^2=\pi_p$. By assumption, we can parametrize any extrinsically flat surface as a ruled surface. Let $\alpha :I\rightarrow M$ be a curve parametrized by arc-length $u$ such that $\alpha(0)=p\in M$ and let $Z$ be a unit vector field along $\alpha$. Now, consider $\Sigma^2$ parametrized by  $\psi(u,v)=\exp_{\alpha(u)}(vZ(u))$. Then, $\psi(0,0)=p\in\Sigma^2$, $\psi(u,0)=\alpha(u)$, and $\psi(\cdot ,v)=\gamma(v)$ are the rulings of $\Sigma^2$, which are geodesics of $M^3$. Note that, in addition to $\langle\partial_2,\partial_2\rangle=1$, we can also assume $\langle\partial_1,\partial_2\rangle\vert_{(u,0)}=0$ and $\langle\partial_1,\partial_1\rangle\vert_{(u,0)}=1$, i.e., $\langle\alpha',Z\rangle=0$ along $\alpha$, where $\partial_1=\partial/\partial u$ and $\partial_2=\partial/\partial v$ are the coordinate tangent vectors. 

If $R$ denotes the curvature tensor of $M^3$, then using that $[\partial_1,\partial_2]=0$ and  $\nabla_{\partial_2}{\partial_2}=0$, it follows that
\begin{equation}
    R_p(\partial_2,\partial_1)\partial_2=\nabla_{\partial_1}\nabla_{\partial_2}{\partial_2}-\nabla_{\partial_2}\nabla_{\partial_1}{\partial_2}+\nabla_{[\partial_1,\partial_2]}\partial_2 = -\nabla_{\partial_2}\nabla_{\partial_1}{\partial_2}.
\end{equation}
Since $\langle\partial_2,\partial_2\rangle=1$, then $\langle\nabla_{\partial_1}\partial_2,\partial_2\rangle=0$ and, therefore, $\nabla_{\partial_1}\partial_2=A\partial_1+B\xi$ for some functions $A$ and $B$, where $\xi$ is the unit normal of $\Sigma^2$.

To finish the proof, we will necessitate the two facts below:
\begin{fact} 
\emph{If  $\beta :I\rightarrow N\subset M$ is a geodesic of $M$, then $\beta$ is an asymptotic line.}
\end{fact} 

Indeed, if $\xi$ is the unit normal of $N^2$, then $\kappa_n(\beta')=\langle\nabla_{\beta'}\beta',\xi \rangle=0$ and $\beta$ is an asymptotic line.

\begin{fact}
\emph{If $K_{ext}(p)=0$, then every direction is an asymptotic direction if $p$ is a planar point of $N^2$. However, if $p$ is a parabolic point of $N^2$, then there exists a unique asymptotic direction at $p$ and, in addition, it is also a principal direction.}
\end{fact}

Indeed, let $\mathbf{e}_1$ and $\mathbf{e}_2$ be the principal directions at $p$ with principal curvatures $\kappa_1(p)$ and $\kappa_2(p)$, respectively. If $p$ is a planar point, then $\kappa_1(p)=\kappa_2(p)=0$ and, therefore, $\kappa_n(V)\equiv 0$ for all $V\in T_p N^2$. If $p$ is a parabolic point, we may assume that $\kappa_2(p)=0$ but $\kappa_1(p)\neq0$. For every $V=\cos\theta\,\mathbf{e}_1+\sin\theta\,\mathbf{e}_2$ in $T_p N^2$, we have $\kappa_n(V)=\kappa_1\cos^2\theta+\kappa_2\sin^2\theta=\kappa_1\cos^2\theta,$ 
which vanishes if and only if $\cos\theta=0$, i.e., if and only if $V=\pm\,\mathbf{e}_2$.
\newline

Let us complete the proof of the theorem. We have seen that $\nabla_{\partial_1}\partial_2=A\partial_1+B\xi$. Using that $[\partial_1,\partial_2]=0$ implies $\nabla_{\partial_1}\partial_2=\nabla_{\partial_2}\partial_1$, it follows that
\[
    B = \langle \xi,\nabla_{\partial_1}\partial_2\rangle = \langle \xi,\nabla_{\partial_2}\partial_1\rangle = \langle -\nabla_{\partial_2}\xi,\partial_1\rangle=0,
\]
where in the last equality we used that, from Facts 1 and 2, it is valid $-\nabla_{\partial_2}\xi=0$. Consequently, we can write
\begin{equation}\label{eq::Ddel1del2EqualAdel1}
    \nabla_{\partial_1}\partial_2=A\,\partial_1.
\end{equation}
Thus, we can write
\begin{equation}
   \nabla_{\partial_2}\nabla_{\partial_1}\partial_2= \partial_2(A)\partial_1+A\nabla_{\partial_2}\partial_1=\partial_2(A)\partial_1+A\nabla_{\partial_1}\partial_2 = (\partial_2A+A^2)\partial_1. 
\end{equation}
Finally, the sectional curvature of $M^3$ at $\pi_p=\mbox{span}\{\partial_1(p),\partial_2(p)\}$ is
\begin{equation}
    K_p(\partial_1,\partial_2)=\dfrac{\langle R(\partial_2,\partial_1)\partial_2,\partial_1\rangle}{\langle\partial_1,\partial_1\rangle\langle\partial_2,\partial_2\rangle-\langle\partial_1,\partial_2\rangle^2}{\Big\vert_p}=-\left(\frac{\partial A}{\partial v}+A^2\right)\Big\vert_p.
\end{equation}
From the form of $K_p(\partial_1,\partial_2)$ and the arbitrariness of $\pi_p$, the sectional curvature is a function of $p$ only. Thus, from the Schur Theorem \cite{doCarmo1992}, $M^3$ is a space form.
\qed \end{proof}

\begin{remark}
We see from the proof above that for a ruled surface in a generic manifold with rulings orthogonal to the directrix, $K_{ext}=0$ implies that the field $Z$ minimizes rotation along $\alpha$. (Please, see Section 5 for a further discussion of this point.)
\end{remark}

\subsection{Application: extrinsically flat surfaces in product manifolds}

Now, we shall apply the results found in this section to the study of extrinsically flat surfaces in the product manifolds $\mathbb{S}^2\times\mathbb{R}$ and $\mathbb{H}^2\times\mathbb{R}$. In these manifolds, the unit vector field $\partial_t$ tangent to $\mathbb{R}$ plays a prominent role. A surface $\Sigma^2$ is said to be a \emph{constant angle surface} if its unit normal $\xi$ makes a constant angle with $\partial_t$. It is known that constant angle surfaces in product manifolds are extrinsically flat \cite{D+07,DM09}. We may naturally ask whether constant angle surfaces are also ruled surfaces. From Prop. \ref{Prop::R1223zeroIFFruled}, we know it is possible to read this information from the curvature tensor. 

\begin{proposition}\label{prop::CteAnglSurfAreRuled}
Let $\Sigma^2$ be a constant angle surface in $\mathbb{S}^2\times\mathbb{R}$ or $\mathbb{H}^2\times\mathbb{R}$, then $\Sigma^2$ is a ruled surface.
\end{proposition}
\begin{proof}
Let us do the proof for $M^3=\mathbb{H}^2\times\mathbb{R}$ equipped with the usual product metric $g=g_R+g_H$.  Consider a surface $\Sigma^2$ with unit normal $\mathbf{e}_3=\xi$.  We may decompose the unit vector field $\partial_t$ as $\partial_t=\sin\theta\, \mathbf{e}_1+\cos\theta\,\mathbf{e}_3$, where $\mathbf{e}_1$ is a unit tangent vector and $\theta$ is the angle between $\partial_t$ and $\xi$, and let $\mathbf{e}_2$ be a unit tangent vector orthogonal to $\mathbf{e}_1$. The curvature tensor of $M^3$ is given by $R^M(X,Y,Z,W) = g_H(X^H,Z^H)g_H(Y^H,W^H)-g_H(X^H,W^H)g_H(Y^H,Z^H)$ \cite{DM09}, where $X^H$ denotes the projection of  $X$ to the tangent space of $\mathbb{H}^2$, while the $\mathbb{R}$-projection is denoted by $X^R=X-X^H$. Let us compute the coefficient $R_{2113}^M$ (\footnote{To apply Prop. \ref{Prop::R1223zeroIFFruled} the repeated index must correspond to the direction of the asymptotic curves, in which case the proposition then tells us whether they are geodesics of the ambient space or not. Here, we are following the notation of \cite{DM09} for the choice of $\mathbf{e}_1$ and $\mathbf{e}_2$, see, e.g., the Eq. (7) of \cite{DM09}.}). First, we have $g_H(e_1^H,\mathbf{e}_2)=g(\mathbf{e}_1,\mathbf{e}_2)-g_R(e_1^R,0)=0$, where we used that $g(\mathbf{e}_i,\mathbf{e}_j)=\delta_{ij}$ and that $\mathbf{e}_2=e_2^H$. In addition, we have
$g_H(\mathbf{e}_2,e_3^H)=g(\mathbf{e}_2,\mathbf{e}_3)-g_R(e_2^R,e_3^R)=g(\mathbf{e}_2,\mathbf{e}_3)=0$. Therefore, $R^M_{2113}=0$ and Prop. \ref{Prop::R1223zeroIFFruled} implies constant angle surfaces are ruled.
\qed \end{proof}

This proposition is compatible with Theorem 3.2 of \cite{DM09} and Theorem 2 of \cite{D+07}, but note that not every ruled surface makes a constant angle with $\partial_t$. In fact, the $\mathbb{R}$-coordinate of such surfaces does not depend on the parameter of the generating curve. Now, we prove that there exist extrinsically flat surfaces that are not of constant angle.

\begin{theorem}\label{Thr::ExistsExtFlatSurfNonCteAngleInH2RandS2R}
There exist extrinsically flat surfaces in $\mathbb{H}^2\times\mathbb{R}$ and $\mathbb{S}^2\times\mathbb{R}$ which do not make a constant angle with $\partial_t$. In addition, the principal direction with vanishing principal curvature must be orthogonal to $\partial_t$.
\end{theorem}
\begin{proof}
Since there exists a constant angle surface tangent to any given 2-plane, every constant angle surface is ruled, and neither $\mathbb{H}^2\times\mathbb{R}$ nor $\mathbb{S}^2\times\mathbb{R}$ is a space form, then Theorem \ref{Theor::CharSpcFraRuledSurf} implies that there exists an extrinsically flat surface which is not ruled and, consequently, such surface can not be of constant angle. 

Concerning the direction $\mathbf{e}_1$ of the principal curves with vanishing principal curvature, if it were $g(\mathbf{e}_1,\partial_t)\not=0$. Then, we would conclude that $R^M_{2113}=0$, which implies that the corresponding extrinsically surface is ruled (see calculations in the proof of Prop. \ref{prop::CteAnglSurfAreRuled}). Therefore, we must have $g(\mathbf{e}_1,\partial_t)=0$. 
\qed \end{proof}

\begin{remark}
Again, consider in $\mathbb{H}^2\times\mathbb{R}$ the decomposition $\partial_t=\sin\theta\,\mathbf{e}_1+\cos\theta\,\mathbf{e}_3$ and the 2-plane with normal $\mathbf{e}_3$ spanned by $\{\mathbf{e}_1,\mathbf{e}_2\}$, where $g(\mathbf{e}_i,\mathbf{e}_j)=\delta_{ij}$. We have seen that $R^M_{2113}=0$. Now, let us compute $R^M_{1223}$. First, $g_H(e_1^H,\mathbf{e}_2)=g(\mathbf{e}_1,\mathbf{e}_2)-g_R(e_1^R,0)=0$. Using $\mathbf{e}_1=\sin\theta\,\partial_t+\cos\theta\,e_1^H$ and $\mathbf{e}_3=\cos\theta\,\partial_t+\sin\theta\,e_3^H$ give  $g_H(e_1^H,\mathbf{e}_3^H)=g(\mathbf{e}_1,\mathbf{e}_3)-g_R(e_1^R,e_3^R)=-\frac{1}{2}\sin2\theta$. Finally, we conclude that $R_{1223}^M=\frac{1}{2}\sin2\theta$ and, therefore, unless the surface is a vertical cylinder over a curve in $\mathbb{H}^2$ or a copy of $\mathbb{H}^2$, i.e., $\theta=0$ or $\theta=\frac{\pi}{2}$, respectively, an extrinsically flat surface with principal curve tangent to $\mathbf{e}_2$, which is orthogonal to $\partial_t$, can not be ruled. This confirms the second part of Thr. \ref{Thr::ExistsExtFlatSurfNonCteAngleInH2RandS2R} concerning the asymptotic curves of extrinsically flat surfaces which {are} not of constant angle.
\end{remark}

\section{Concluding remarks}

For Riemannian space forms, we filled an existing gap in the literature concerning
the fundamentals of ruled surfaces. Namely, we introduced the concepts of striction
curves and distribution parameter and computed the first and second fundamental
forms. As in Euclidean space, these concepts allowed us to locate the singularities
of ruled surfaces as well as characterize the vanishing of the extrinsic Gaussian
curvature.

It is well known that extrinsically flat surfaces in space forms must be foliated by
geodesics. We provided a model-independent proof for this result, which allowed us
to identify the necessary and sufficient condition the curvature tensor must satisfy in
order to guarantee the validity of this implication: namely, $R_{trrn}=0$, where $n$ refers
to the direction normal to the surface, $r$ refers to the direction of the rulings, and $t$
refers to the remaining tangent direction. Further, we characterized space forms as
the only manifolds with the properties that there exist an extrinsically flat surface
tangent to any 2-plane and that these surfaces are all ruled. This result is part of our
ongoing effort to find characteristic properties of space forms \cite{dSdS20}. As an application
of our findings, we proved that there must exist extrinsically flat surfaces in the
product manifolds $\mathbb{H}^2\times\mathbb{R}$ and $\mathbb{S}^2\times\mathbb{R}$ that do not make a constant angle with the
real direction. This opens the problem of actually constructing such surfaces. Since
product manifolds are not space forms, there is also the quest of finding the properties
an extrinsically flat surface in such spaces must satisfy in order to be ruled.
The theory of ruled surfaces in product manifolds is currently under investigation
and shall be the content of a follow-up publication. It would be desirable to better


Finally, we want to point the similarity between the proofs of Thr. \ref{Theor::CharSpcFraRuledSurf} and of Thr.  2 of \cite{dSdS20}. Namely,  Eq. \eqref{eq::Ddel1del2EqualAdel1} which is crucial in our proof plays a role similar to that of  the umbilicity condition  $\nabla_XY=-\lambda X$ in \cite{dSdS20}. This similarity may be interpreted in terms of rotation minimizing vector fields. More precisely, our results show that in space forms we have $K_{ext}\propto (\alpha',Z,\nabla_{\alpha'}Z)$, which implies that for $Z$ {orthogonal} to $\alpha$, minimizing rotation is an if-and-only-if condition for a ruled surface to be extrinsically flat. As the implication ``$K_{ext}=0\Rightarrow $ ruled" is characteristic of constant curvature, it is reasonable to see the minimizing rotation property playing a role. In a generic Riemannian manifold, $K_{ext}$ of ruled surfaces is always proportional to $h_{12}=\langle\nabla_{\partial_1}\partial_2,\xi\rangle$. Then, when the rulings are normal to $\alpha$, $K_{ext}=0$ \emph{along the directrix} if and only if $Z=\partial_2\vert_{\alpha}$ minimizes rotation. This suggests that in general $K_{ext}\propto (\alpha',Z,\nabla_{\alpha'}Z)$ at least along $\alpha$. It remains to verify whether this is also true for the remaining points of the ruled surface or whether $K_{ext}\propto (\alpha',Z,\nabla_{\alpha'}Z)$ provides yet another characteristic property of constant curvature. 

\begin{acknowledgements}
The authors would like to thank useful discussions with Jos\'e Alan Farias dos Santos (Recife, Brazil). In addition, LCBdS would like to thank the financial support provided by the Mor\'a Miriam Rozen Gerber Fellowship for Brazilian Postdocs and the Faculty of Physics Postdoctoral Excellence Fellowship.
\newline
\textbf{Conflict of interest.} The authors declare they have no conflict of interest.
\end{acknowledgements}


\end{document}